%% file: algebraization.tex
\documentclass[12pt]{amsart}
\usepackage{amsmath}
\usepackage{amssymb}
\usepackage{amsthm}
\usepackage{amscd}
\usepackage{amsfonts}
\usepackage{graphicx}%
\usepackage{fancyhdr}
\usepackage{enumerate}
\usepackage{mathrsfs}
\usepackage{hyperref}
\usepackage[all]{xy}

\input{definitions.tex}

\begin{document}

\title{Algebraic Frobenius Splitting of Cotangent Bundles of Flag Varieties}
\author{Chuck Hague}
\email{hague@math.udel.edu}
\address{Department of Mathematical Sciences \\
University of Delaware \\
501 Ewing Hall \\
Newark, DE 19716}
\begin{abstract} Following the program of algebraic Frobenius splitting begun by Kumar and Littelmann, we use representation-theoretic techniques to construct a Frobenius splitting of the cotangent bundle of the flag variety of a semisimple algebraic group over an algebraically closed field of positive characteristic. We also show that this splitting is the same as one of the splittings constructed by Kumar, Lauritzen, and Thomsen.
\end{abstract}

\maketitle

\setcounter{tocdepth}{2}
\tableofcontents

\section{Introduction}

\subsection{Background}

Let $G_k$ be a semisimple, simply-connected algebraic group over an algebraically closed field $k$ of positive characteristic $p$ and let $B_k \subseteq G_k$ be a Borel subgroup. We assume that $p$ is a good prime for $G$ (cf Definition \ref{def:good prime}).  One of the fundamental results of the theory of Frobenius splitting (\cite{MR85}) is that the flag variety $G_k/B_k$ is Frobenius split. In the papers \cite{KL00} and \cite{KL02}, Kumar and Littelmann use the quantum Frobenius morphism and a variant of its splitting, both due to Lusztig \cite{L90}, to construct an alternate proof of the splitting of $G_k/B_k$ using purely representation-theoretic constructions; they call this an algebraization of Frobenius splitting.

More precisely, Kumar and Littelmann construct morphisms between induced representations for hyperalgebra and quantum group representations. Upon base change, these morphisms can be identified with morphisms on the structure sheaf $\struct C$ of an affine cone $C$ over $G_k/B_k$. In particular, the quantum Frobenius morphism induces the $p^{th}$ power on $\struct C$ and the quantum splitting morphism induces a splitting of the $p^{th}$ power morphism on $\struct C$. This implies that $C$ is Frobenius split and hence by a process of sheafification that $G_k/B_k$ is Frobenius split as well.

Gros and Kaneda \cite{GK10} then showed the argument of Kumar-Littelmann can be simplified; in particular, one does not have to go to the level of quantum groups. Instead, all of the constructions of \cite{KL00} and \cite{KL02} can be done purely on the level of hyperalgebras. In particular, they construct a morphism $\varphi$ which is the hyperalgebra version of the quantum splitting morphism. In this paper, we use the constructions in \cite{GK10} to continue the Kumar-Littelmann program of algebraic Frobenius splitting and give a purely representation-theoretic proof that the cotangent bundle $\co$ of $G_k/B_k$ is Frobenius split, a fact which was first proved by geometric means in \cite{KLT}.

One main advantage of using algebraic Frobenius splitting techniques is that one can concretely write down the splitting. In particular, the hope is that using the algebraic method will make it easier to check that certain subvarieties are compatibly split.

\subsection{Algebraic Frobenius splitting}

Let $X$ be a projective $k$-variety and let $\L$ be an ample line bundle on $X$. Set
\begin{equation}
R_\L := \bigoplus_{n \geq 0} \cohom 0 X {\L^n} \,,
\end{equation}
the affine cone over $X$ corresponding to $\L$. The main fact in algebraic Frobenius splitting (Lemma 1.1.14 in \cite{BK}) is that $X$ is Frobenius split if and only if $\tr{Spec}( R_\L )$ is. In turn, $\tr{Spec}( R_\L )$ is Frobenius split if and only if $R_\L$ is a Frobenius split $k$-algebra: i.e., there exists an $\fp$-linear endomorphism $s$ of $R_\L$ such that (1) $s(f^p g) = f \cdot s(g)$ for all $f, g \in R_\L$ (this is called \tb{Frobenius-linearity} of $s$) and (2) $s(f^p) = f$ for all $f \in R_\L$.

We now apply these ideas to the case $X = \pco$, the projectivization of the cotangent bundle $\co$. Let $U_k \subseteq B_k$ be the unipotent radical of $B_k$ and let $U^-_k$ be the opposite unipotent radical. Let $pr : \co \to G_k/B_k$ be the projection and set $F_k := pr^{-1}( U^-_k B_k ) \subseteq \co$, the fiber over the big cell $ U^-_k B_k \subseteq G_k / B_k $. Then $F_k$ is an affine subvariety of $\co$ isomorphic to $ U^-_k \times U_k $.

Let $G$ be a split form of $G_k$ over $\fp$. We first construct, for any weight $\lambda$ of $G$, a polynomial ring $\Rh \lambda$ over $\fp$ such that $ \Rh \lambda \otimes_\fp k \cong k[F_k] $. This ring carries an action of the hyperalgebra of $G$; taking the locally finite part gives a ring $\R \lambda$. When $\lambda$ is a regular dominant weight, $\R \lambda \otimes_\fp k$ is isomorphic to $R_\L$ for a very ample bundle $\L$ on $\pco$. Further, upon base change to $k$ the natural inclusion $\R \lambda \hookrightarrow \Rh \lambda$ corresponds to the inclusion $ R_\L \hookrightarrow k[F_k] $.

Now, since $\pco$ is split if and only if $\co$ is, it suffices to construct a splitting of the $k$-algebra $R_\L$. To this end, we first work over $\fp$ and construct a splitting $\tot$ of $\Rh \lambda$ that restricts to a splitting of the subalgebra $\R \lambda$. Upon base change, this induces a splitting of $R_\L$. Geometrically, this corresponds to a splitting of the ring $ k[F_k] $ (or, equivalently, a splitting of the affine scheme $F_k$) that restricts to a splitting of the subring $R_\L$.

\subsection{Details}
We now give more details on the construction of the rings $\Rh \lambda$ and $\R \lambda$ and the splitting morphism $\tot$. As above let $G$ be a split form of the group $G_k$ over $\fp$ and let $T \subseteq G$ be a split maximal torus. Let $B \subseteq G$ be a Borel subgroup of $G$ containing $T$. Let $B^-$ denote the opposite Borel subgroup. Let $U \subseteq B$ and $U^- \subseteq B^-$ be the respective unipotent radicals. We consider the root spaces of $B$ to correspond to the positive roots. Let $\Lambda$ denote the weight lattice of $T$.

Let $\barn$ denote the hyperalgebra of $U$. The torus-locally finite part $\barnd$ of the full linear dual of $\barn$ is naturally isomorphic to $\fp[U]$, the coordinate ring of $U$. Set $\n := \tr{Lie}(U)$; then a Springer isomorphism $ U \isom \n $ induces a $B$-equivariant isomorphism $ \fp[U] \isom \fp[\n] $ and hence a $B$-equivariant isomorphism $\barnd \isom \fp[\n]$. Since $\fp[\n]$ has a natural $B$-equivariant grading by polynomial degree, we obtain a $B$-equivariant grading $\barndg n$ on $\barnd$.

In \S\ref{sub:algebraic constructions} we construct, for each $\lambda \in \Lambda$, the $\fp$-algebras $\Rh \lambda$ and $\R \lambda$. These rings are defined by inducing (twists of) the $B$-modules $\barndg n$ to $\barg$-modules. We can interpret this construction in the following way. The rings $\Rh \lambda$ are all isomorphic to polynomial rings (cf the proof of Proposition \ref{pr:KLT} below). In particular they are all naturally isomorphic to the ring of functions on $U^- \times U$. Base changing to $k$, $\Rh \lambda \otimes_\fp k$ is isomorphic to the ring of functions on the affine space $F_k$ defined above. Different choices of $\lambda \in \Lambda$ give rise to different $\barg$-algebra structures on this polynomial ring, so the rings $\Rh \lambda$ give a family of $\bargk$-module structures on $k[F_k] \cong k[U^-_k] \otimes k[U_k]$, where $\bargk$ is the hyperalgebra of $G_k$. Taking the $\barg$-locally finite part of $\Rh \lambda$ gives the ring $\R \lambda$. Remark that the rings $\R \lambda$ are \emph{not} all isomorphic for various choices of $\lambda \in \Lambda$.

Motivated by \cite{KLT}, the splitting $\tot$ of $\Rh \lambda$ is constructed via the trace methodology described as follows. Given a polynomial ring $P$ and a choice of algebra generators of $P$ there is a Frobenius-linear trace morphism $\trace$ on $P$, and every Frobenius-linear endomorphism of $P$ is of the form
\begin{equation}
f \mapsto \trace( f \cdot g ) 
\end{equation}
for some fixed $g \in P$. If $Q \subseteq P$ is a subring we can look for $q \in Q$ such that (1) $\trace(f \cdot q) \in Q$ for all $f \in Q$ and (2) $\trace( - \cdot q )$ is a Frobenius splitting of $P$. This will give a Frobenius splitting of the ring $Q$. 

In particular, since $\Rh \lambda$ is a polynomial ring we have a Frobenius-linear trace map $\trace$ on $\Rh \lambda$ corresponding to an appropriate choice of $\fp$-algebra generators of $\Rh \lambda$ (cf \S\ref{sub:spl and the trace map}). We apply the trace methodology to the subring $\R \lambda \subseteq \Rh \lambda$. In these constructions we first work over $\fp$ and then base-change to $k$ later.

In \S\ref{sub:spl} we construct, using representation-theoretic techniques, a Frobenius-linear endomorphism $\spl$ of $ \Rh \lambda $ which turns out (\S \ref{sub:spl and the trace map}) to be the same as the trace morphism $\trace$. In \S\ref{sub:mul} we construct an element $\psi_{f_+ \otimes f_-} \in \R \lambda$ for $\lambda = 0$ and in \S\ref{sub:tot} we show that the Frobenius-linear endomorphism
\begin{equation} \label{eq:tot introduction}
\tot : f \mapsto \spl( \psi_{f_+ \otimes f_-} \cdot f )
\end{equation}
of $\Rh \lambda$ is a Frobenius splitting that preserves $\R \lambda$. In particular, $\tot$ restricts to a Frobenius splitting of $\R \lambda$ as desired. (Remark that below we write $\mulf$ for multiplication by $\psi_{f_+ \otimes f_-}$ and hence, concisely, $ \tot = \spl \circ \mulf $).

In \S\ref{sec:Frobenius} we base-change to $k$ and construct the desired splitting of $\pco$ and hence obtain a splitting of $\co$. We also show that this splitting is the same as one of the homogeneous splittings of $\co$ in \cite{KLT}.

Also, I would like to thank Shrawan Kumar and George McNinch for helpful conversations and an anonymous referee for pointing out typos and areas for improvement.

\section{Algebraic splitting} \label{sec:algebraic splitting}

\subsection{Setup}

Throughout \S2 we assume all algebraic groups, algebras, schemes, vector spaces, etc. are over $\fp$. Recall the groups $G$, $B$, $U$, $T$, etc. from above.

\subsubsection{} \label{subsub:background 1}

\begin{definition} \label{def:good prime} We say that a prime $p$ is \tb{bad} for a simple algebraic group $G$ in the following cases. If $G$ is of type $A_\ell$ then no prime is bad; if $G$ is of type $B_\ell$, $C_\ell$, or $D_\ell$ then $p=2$ is bad; if $G$ is of type $ E_6, E_7, F_4 $, or $G_2$ then $p = 2,3$ are bad; and if $G$ is of type $E_8$ then $p = 2, 3, 5$ are bad. We say that $p$ is a bad prime for a semisimple algebraic group $G$ if it is bad for any of its simple components, and we say that $p$ is a \tb{good} prime for $G$ if it is not bad.
\end{definition}

From here on we assume that $p$ is a good prime for $G$.

For an algebraic group $H$ over $\fp$ let $I \subseteq \fp[H]$ denote the ideal of the identity element. The subspace of the linear dual of $\fp[H]$ consisting of elements that vanish on some power of $I$ is called the hyperalgebra of $H$; it has a natural Hopf algebra structure obtained from the Hopf algebra structure on $\fp[H]$. Let $ \barg $, $ \barb $, $ \barbm $, $ \barn $, $ \barnm $, and $ \bart $ denote the hyperalgebras of $G$, $B$, $B^-$, $U$, $U^-$, and $T$, respectively.

The Frobenius morphism $\fp[G] \to \fp[G]$, $f \mapsto f^p$ induces a morphism $\fr : \barg \to \barg$ of $\fp$-algebras. We will denote the restriction of $\fr$ to $ \barb$, $ \barn $, etc by $\fr$ as well. Let $\ell$ denote the rank of $G$. $\barg$ is generated by elements $E_i^{(n)} \in \barn$, $F_i^{(n)} \in \barnm$, and $ {H_i \choose n} \in \bart $ for $n \geq 0$ and $1 \leq i \leq \ell$. On these generators, we have: 
\begin{subequations}
\begin{equation}
\fr( E_i^{(n)} ) = \left \{ \begin{array}{ll} E_i^{(n/p)} & \tr{if } p \mid n \\
0 & \tr{if } p \nmid n \end{array} \right.
\end{equation}
\begin{equation}
\fr( F_i^{(n)} ) = \left \{ \begin{array}{ll} F_i^{(n/p)} & \tr{if } p \mid n \\
0 & \tr{if } p \nmid n \end{array} \right.  
\end{equation}
and
\begin{equation}
\fr{ H_i \choose n } = \left \{ \begin{array}{ll} { H_i \choose n/p } & \tr{if } p \mid n \\
0 & \tr{if } p \nmid n \end{array} \right.  \, .
\end{equation}

\end{subequations}

\subsubsection{}

By \cite{KL00} and \cite{L90} we have $\fp$-algebra morphisms $\frp : \barn \to \barn$, $ \frpm : \barnm \to \barnm $, and $\frpt : \bart \to \bart$ given by
\begin{subequations}
\begin{equation}
\frp( E_i^{(n)} ) = E_i^{(pn)} \,, 
\end{equation}
\begin{equation}
\frpm( F_i^{(n)} ) = F_i^{(pn)} \,,
\end{equation}
and
\begin{equation}
\frpt \displaystyle \binom{H_i} n = \binom{H_i}{pn}
\end{equation}
for all $1 \leq i \leq \ell$ and $n \geq 0$.
\end{subequations}

Set
\begin{equation} \label{eq:mu0}
\mu_0 := \displaystyle \prod_{i=1}^\ell \binom{H_i - 1}{p-1} = \prod_{i=1}^\ell (1 - H_i^{p-1}) \, , \end{equation}
an idempotent in $\bart$. By \cite{GK10} Theorem 1.4, there is a multiplicative morphism
\begin{subequations}
\begin{equation}
\varphi : \barg \to \barg 
\end{equation}
given by
\begin{equation}
\varphi(YHX) = \frpm Y \cdot \frpt H \cdot \frp X \cdot \mu_0 
\end{equation}
\end{subequations}
for all $Y \in \barnm$, $H \in \bart$, and $X \in \barn$. Further, $\mu_0$ commutes with all elements in the image of $\varphi$, so if we consider $ \im \, \varphi $ as an $\fp$-algebra with unit $\mu_0$, then $\varphi $ is an $\fp$-algebra morphism.

Note that $$\fr (H_i) = \fr \binom{ H_i } 1 = 0 \,.$$ Hence $ \fr( H_i^{p-1} ) = 0 $ which implies $ \fr(\mu_0) = 1 $, and we have the following important fact:
\begin{equation}
\fr \circ \varphi = \id_\barg \, .
\end{equation}

Let $\Lambda$ denote the weight lattice of $G$. For $\lambda \in \Lambda$ let $c_\lambda : \bart \to \fp$ be the character associated to $\lambda$. We have the following result from \cite{GK10}.

\begin{subequations}

\begin{lemma} (Lemme 2.1 in \cite{GK10}) \label{lem:varphi and weights} For all $\lambda \in \Lambda$ we have
\begin{equation} \label{eq:varphi and weights} c_\lambda \circ \varphi|_\bart = \left\{
\begin{array}{ll}
c_{\lambda/p} & \tr{if $\lambda \in p \Lambda$} \\
0 & \tr{if $\lambda \notin p \lambda$.}
\end{array}
\right.
\end{equation}
\end{lemma}

In particular,
\begin{equation} \label{eq:mu0 and weights}
c_\lambda(\mu_0) = \left\{
\begin{array}{ll}
1 & \tr{if $\lambda \in p \Lambda$} \\
0 & \tr{if $\lambda \notin p \lambda$.}
\end{array}
\right.
\end{equation}

\end{subequations}

\subsection{Algebraic constructions and preliminaries} \label{sub:algebraic constructions}
\subsubsection{}

For a Hopf algebra with comultiplication $\Delta$ we use the Sweedler notation
\begin{align*}
\Delta X &= \sum X_{(1)} \otimes X_{(2)} \,, \\
\big( (\Delta \otimes \id \big) \circ \Delta)(X) &= \sum X_{(1)} \otimes X_{(2)} \otimes X_{(3)} \,, 
\end{align*}
etc. Let $\epsilon$ and $\sigma$ denote the augmentation and coinverse of $\barg$, respectively. By a slight abuse of notation we will also use the same notation for the various sub-Hopf algebras $\barn$, $\barnm$, etc of $\barg$.

For any $ \bart $-module $V$ (resp. $\barg$-module $W$) let $\finh V$ (resp. $\fing W$) denote the $\bart$ (resp. $\barg$)-locally finite part of $V$ (resp. $W$). Also set $\d V := \finh \, V^*$. If $V$ is a module for $ \barg $, $ \barb $, or $ \barbm $ then so is $\d V$.

Recall that for a Hopf algebra $H$ and algebra $A$, we say that $A$ is an \tb{$H$-module algebra} if $A$ is an $H$-module and
\begin{equation}
h.(ab) = \sum (\sw h 1.a) \cdot (\sw h 2.b)
\end{equation}
for all $h \in H$ and $a, b \in A$.

We have the conjugation (or adjoint) $\barb$-action on $\barn$ given by
\begin{equation}
X*Y = \sum X_{(1)} Y \sigma (X_{(2)}) \,, 
\end{equation}
where $\sigma$ is the coinverse. This action induces a dual action of $\barb$ on $\barnd$, also denoted by $*$. Under the adjoint action, $\barn$ and $\barnd$ become $\barb$-module algebras. From here on, we consider $ \barn $ as a $\barb$-module under the $*$-action.

There is a duality pairing between $ \fp[U] $ and $\barn$ which defines the Hopf algebra structure on $\barn$ (cf \S I.7 in \cite{Ja03}). There is a natural Hopf algebra structure on $\barnd$ obtained from duality with $\barn$ and hence a Hopf algebra isomorphism $ \fp[U] \cong \barnd $. This is also an isomorphism of $\barb$-module algebras, where we take the $\barb$-action on $\fp[U]$ induced by the conjugation action of $B$ on $U$.

\subsubsection{} \label{subsub:Springer isom}

Recall that we are assuming that $p$ is a good prime for $G$. By \cite{SpUnip}, Proposition 3.5, there is a $B$-equivariant Springer isomorphism $ U \cong \n $ which intertwines the conjugation $B$-action on $U$ with the standard $B$-action on $\n$. (There are in fact infinitely many Springer isomorphisms, so let us fix any one of them). Thus we obtain isomorphisms of $\barb$-module algebras
\begin{equation} \label{eq:the 3 barb-module algebras}
\barnd \cong \fp[U] \cong \fp[\n] \cong \symnd \,.
\end{equation}

As $\symnd$ has a natural $\barb$-equivariant algebra grading, this induces a $\barb$-equivariant multiplicative grading $ \barndg n $ on $ \barnd $. Dually, we obtain a $\barb$-equivariant grading $ \barng n $ on $\barn$ such that the comultiplication $\Delta : \barn \to \barn \otimes \barn$ is gradation-preserving under the induced grading on $ \barn \otimes \barn $.

\begin{remark}
For all of the proofs below, we only use the fact that there is a $\barb$-module algebra isomorphism $ \barnd \cong \symnd $; hence we could use any such isomorphism. In particular, instead of a Springer isomorphism, we could use the isomorphism constructed in \cite{FP}. Different choices of isomorphisms may, however, result in different splittings.
\end{remark}

\subsubsection{Induction functors and duality}  \label{subsub:background 2}

Let $M$ be a $B$-module. Then $\Hom_{\barb}( \barg, M )$ has a $\barg$-module structure given by \begin{equation}
(Y.f)(X) = f(XY) \tr{ for all }X, Y \in \barg \tr{ and } f \in \Hom_{\barb}( \barg, M ) \,. 
\end{equation}

For any $B$-module $M$ set
\begin{subequations}
\begin{equation}
\indg M := \fing \, \Hom_{\barb}( \barg, M ) 
\end{equation}
and
\begin{equation}
\indh M := \indhugly M  \, , 
\end{equation}
\end{subequations}
Note that we have inclusions of $\barg$-modules $$  \indg M \subseteq \indh M \subseteq \Hom_\barb( \barg, M )  \, .$$

We will frequently use the following fact. For any $\bart$-locally finite $\barb$-module $M$ we have $\bart$-module isomorphisms
\begin{equation} \label{indh M = hom involving fp}  \indh M \cong \finh \, \Hom_\fp \big(  \barnm, M  \big)  \cong \barnmd \otimes M \, .
\end{equation}

\subsubsection{} \label{subsub:identification of indh barnd}
Consider the group algebra $\fp[\Lambda]$ of the lattice $\Lambda$; then $ \fp[\Lambda] $ is naturally a $\bart$-module algebra. We make it into a $\barb$-module algebra by giving it a trivial $\barn$-action. For each $\lambda \in \Lambda$ let $v_\lambda \in \fp[\Lambda]$ denote the element corresponding to $\lambda$. Then, in particular, we have
\begin{equation}
v_\lambda \cdot v_\mu = v_{\lambda + \mu}
\end{equation}
for all $\lambda, \mu \in \Lambda$. We also identify $\fp. v_0 $ with $\fp$ via the basis element $v_0$. This induces a bilinear pairing
\begin{equation} \label{eq:weight pairing}
\fp . v_\lambda \otimes \fp. v_{-\lambda} \to \fp. v_0 \to \fp
\end{equation}
for all $\lambda \in \Lambda$.

For $\lambda \in \Lambda$ let $\chi_\lambda$ denote the 1-dimensional $ \barb $-module corresponding to the character $\lambda$ of $\bart$ and set \begin{equation} \label{eq:ho} \ho \lambda := \indg{\chi_{-\lambda}} \, ,\end{equation} the induced $G$-module with lowest weight $-\lambda$. In the sequel we will freely identify $\chi_\lambda$ with $\fp.v_\lambda \subseteq \fp[\Lambda]$.

\begin{lemma} Choose $\lambda \in \Lambda$. There is a natural $\barg$-equivariant inclusion
\begin{equation} \label{eq:identification of indh barnd}
\Indh{ \barnd \otimes \chi_{-\lambda} } \hookrightarrow \big( \barg \otimes \barn \otimes \chi_\lambda \big)^* \, ,
\end{equation}
where the $\barg$-action on $\big( \barg \otimes \barn \otimes \chi_\lambda \big)^*$ is given by
\begin{equation} \label{eq:barg action}
(Z.f)(X \otimes Y \otimes v_\lambda) = f(XZ \otimes Y \otimes v_\lambda)
\end{equation}
for all $X, Z \in \barg$ and $Y \in \barn$.

Further, the image of the inclusion (\ref{eq:identification of indh barnd}) consists of the $\bart$-locally finite $f \in \big( \barg \otimes \barn \otimes \chi_\lambda \big)^*$ such that
\begin{equation} \label{eq:barb equivariance}
f(AX \otimes Y \otimes v_\lambda) = f \big( X \otimes \sigma A*(Y \otimes v_\lambda) \big)
\end{equation} for all $A \in \barb$.
\end{lemma}

\begin{proof}
From (\ref{eq:weight pairing}) we can naturally identify $\chi_{-\lambda}$ with $\chi_\lambda^*$. Hence for $f \in \Indh{ \barnd \otimes \chi_{-\lambda} }$ and $X \in \barg$ we can consider $f(X)$ as an element of $\big( \barn \otimes \chi_\lambda \big)^*$. We define the inclusion (\ref{eq:identification of indh barnd}), denoted by $\theta$, as follows: for $f \in \Indh{ \barnd \otimes \chi_{-\lambda} }$, $X \in \barg$, and $Y \in \barn$ set
\begin{equation}
\theta(f)( X \otimes Y \otimes v_\lambda ) = f(X)( Y \otimes v_\lambda ) \,.
\end{equation}
The rest of the statements in the lemma are now straightforward to verify.
\end{proof}

In the sequel, for ease of computation we will frequently use this lemma to identify $ \Indh{ \barnd \otimes \chi_{-\lambda} } $ with its image under the inclusion (\ref{eq:identification of indh barnd}). Remark that (\ref{eq:barb equivariance}) is just the statement that $f$ is $\barb$-linear.

\subsubsection{The algebras $\Rh \lambda$ and $\R \lambda$} \label{subsub:coind barn coalgebra structure}

For any $ \mu, \lambda \in \Lambda $ we have (using the identification (\ref{eq:identification of indh barnd}) above) a $\barg$-equivariant multiplication map
\begin{subequations}
\begin{equation} \label{eq:multiplication in Rh}
\Indh{ \barnd \otimes \chi_{-\mu} } \otimes \Indh{ \barnd \otimes \chi_{-\lambda} } \to \Indh{ \barnd \otimes \chi_{-\mu - \lambda} }
\end{equation}
given by
\begin{equation} \label{eq:multiplication in ih}
(f \cdot g)( X \otimes Y \otimes v_{\mu + \lambda} ) = \sum f( \sw X 1 \otimes \sw Y 1 \otimes v_\mu ) \cdot g( \sw X 2 \otimes \sw Y 2 \otimes v_\lambda ) \, .
\end{equation} Since comultiplication in $\barn$ preserves the gradation, the multiplication map (\ref{eq:multiplication in Rh}) restricts to a degree-preserving map
\begin{equation}
\Indh{ \barndg n \otimes \chi_\mu } \otimes \Indh{ \barndg m \otimes \chi_\lambda } \to \Indh{ \barndg{n+m} \otimes \chi_{\mu + \lambda} }
\end{equation} for all $n, m \geq 0$. 
\end{subequations}

\begin{subequations}
For $\lambda \in \Lambda$ set
\begin{equation}
\Rh \lambda := \bigoplus_{n \geq 0} \ih n \lambda  \, .
\end{equation} By the above, $\Rh \lambda$ is a $\barg$-module algebra. Also set
\begin{equation}
\R \lambda := \fing \, \Rh \lambda = \bigoplus_{n \geq 0} \ig n \lambda  \, .
\end{equation} Since multiplication is $\barg$-equivariant, $\R \lambda$ is a $\barg$-module subalgebra of $\Rh \lambda$.
\end{subequations}

\begin{remark} Note that by (\ref{indh M = hom involving fp}) we have a natural $\fp$-algebra inclusion
\begin{equation}
\Rh \lambda \hookrightarrow \barnmd \otimes \barnd \otimes \fp[\Lambda]
\end{equation}
for all $\lambda \in \Lambda$.
\end{remark}

\subsection{The $p^{th}$ power morphism $\frtstar$} \label{sub:the pth power morphism}

\subsubsection{} \label{subsub:the pth power morphism}

Recall the morphism $\fr$ from \S\ref{subsub:background 1}. Let $\frstar$ (resp. $\frstar^-$) be the endomorphism of $\barnd$ (resp. $\barnmd$) dual to the endomorphism $\fr$ of $\barn$ (resp. $\barnm$). Note that since $\fr$ is a Hopf algebra morphism, so are $\frstar$ and $\frstar^-$.

\begin{lemma} \label{lem:frstar is the pth power morphism}
$\frstar$ (resp. $\frstar^-$) is the $p^{th}$ power morphism on $\barnd$ (resp. $\barnmd$).
\end{lemma}

\begin{proof}
By definition, $\fr$ is dual to the $p^{th}$ power morphism on $\fp[U]$. Since $\barnd \cong \fp[U]$ as $\fp$-algebras (cf (\ref{eq:the 3 barb-module algebras}) above), we have that $\frstar$ is the $p^{th}$ power map on $\barnd$. The statement about $\frstar^-$ is proved similarly.
\end{proof}

\subsubsection{}

Choose $\lambda \in \Lambda$. Since $\frstar$ is the $p^{th}$ power morphism on $\barnd$ it sends $\barndg n$ to $\barndg {pn}$ and we have an endomorphism $\frtstar$ of $\Rh \lambda$ given by the direct sum of the morphisms
\begin{align}
\ih n \lambda & \to \ih {pn} \lambda \, , \nonumber \\
(\frtstar f)(X \otimes Y \otimes v_{pn \lambda}) &= f( \fr X \otimes \fr Y \otimes v_{n \lambda} )
\end{align}
for all $X \in \barg$ and $Y \in \barn$.

\begin{proposition} \label{pr:frtstar is the pth power morphism} $\frtstar$ is the $p^{th}$ power morphism on $R_\lambda^\h$ (and hence restricts to the $p^{th}$ power morphism on $R_\lambda$).
\end{proposition}

\begin{proof} There are natural algebra isomorphisms
\begin{equation} \label{eq:frtstar is the pth power morphism}
\Rh \lambda \cong \bigoplus_{n \geq 0} \bargd \otimes_\barb \left( \barndg n \otimes \chi_{-n \lambda} \right) \cong \bigoplus_{n \geq 0} \barnmd \otimes  \barndg n \otimes \chi_{-n \lambda} \, . 
\end{equation}
The algebra structure on the ring on the right-hand side of (\ref{eq:frtstar is the pth power morphism}) is induced from the algebra structure on $ \barnmd \otimes \barnd $, so it suffices to verify that the endomorphism $ \frstar^- \otimes \frstar $ of $ \barnmd \otimes \barnd $ is the $p^{th}$ power morphism. But this is clear by Lemma \ref{lem:frstar is the pth power morphism}.
\end{proof}

\subsection{The morphism $\spl$} \label{sub:spl}

\subsubsection{The small hyperalgebras}
Set $E_0 :=  \displaystyle \prod_{\beta \in \Delta^+} \EE\beta{p-1}$ and $F_0 :=  \displaystyle \prod_{\beta \in \Delta^+} \FF\beta{p-1}$. By \cite{Hab80}, Proposition 6.7, $E_0$ and $F_0$ are independent of the ordering of the roots. Let $\rho$ denote the half-sum of the positive roots; then $E_0$ (resp. $F_0$) has weight $2(p-1) \rho$ (resp. $-2(p-1) \rho$).

Let $ \smalln $ denote the "small" hyperalgebra associated to $U$, i.e. the sub-Hopf algebra of $\barn$ generated by $ \displaystyle \prod_{\beta \in \Delta^+} \E\beta{m_\beta} $ for $0 \leq m_\beta < p$ (where we take any fixed ordering of $\Delta^+$). Similarly, we have the sub-Hopf algebra $ \smallnm $ of $\barnm$.

\begin{subequations}
Also let $ \smallt $ denote the sub-Hopf algebra of $ \bart $ generated by the elements $ \displaystyle \prod_{i=1}^\ell \binom{H_i}{n_i} $ for $0 \leq n_i < p$. The equality
\begin{equation}
\binom {pn} m = 0 \tr{ for all } n \in \Z \tr{ and } 0 \leq m < p
\end{equation}
in $\fp$ implies
\begin{equation} \label{eq:smallt and p weights}
c_{p \lambda + \mu} ( z ) = c_\mu(z) \tr{ for all } \mu, \lambda \in \Lambda \tr{ and } z \in \smallt \, .
\end{equation}
\end{subequations}

For any Hopf algebra $H$ let $H^+$ denote the augmentation ideal. We have the following useful result.

\begin{lemma}[\cite{Hab80}, Lemmas 6.5 and 6.6 and Proposition 6.7] \label{lem: basic E0 and F0 facts} $E_0$ (resp. $F_0$) is central in $\barn$ (resp. $\barnm$). In particular, $E*E_0 = 0$ and $F*F_0=0$ for all $E \in \barn^+$ and $F \in \barnm^+$. Further, $ E_0 \cdot \smalln^+ = 0 $ and $ F_0 \cdot \smallnm^+ = 0 $.
\end{lemma}

We also need the following technical lemma.

\begin{lemma} \label{lem:useful E0 and F0 facts} \begin{enumerate}[(1)] 
\item $ E_0 \cdot \frp (Z * Y) = E_0 \cdot \big(  \frp Z * \frp Y  \big) $ for all $Y, Z \in \barn$.
\item $E_0 \cdot ( N * X ) = 0$ for all $N \in \smalln^+$ and $X \in \barn$.
\end{enumerate}
\end{lemma}

\begin{proof} (1) Since $\frp$ is an $\fp$-algebra morphism and since
\begin{equation} \label{eq:useful E0 and F0 facts}
E_0 \cdot (A*B) = A * (E_0 B)
\end{equation}
for all $A, B \in \barn$ (by the centrality of $E_0$), it suffices to verify the statement in the case that $Z = \E i m$ for some $1 \leq i \leq \ell$ and $m > 0$. We have:
\begin{align*} E_0 \cdot \big(  (\frp \E i m) * \frp Y  \big) &= E_0 \cdot \big( \E i {pm} * \frp Y \big) \\
&=  \sum_{j=0}^{pm} (-1)^{pm-j} \, E_0 \E i j \frp (Y) \E i {pm - j} \\
&=  \sum_{j=0}^m (-1)^{pm-pj} E_0 \E i {pj} \frp(Y) \E i {pm - pj} \\
& \quad \tr{(by Lemma \ref{lem: basic E0 and F0 facts})} \\
&=  \sum_{j=0}^m (-1)^{m-j} E_0 \frp \big( \E i j Y \E i {m - j} \big) \\
&=  E_0 \cdot \frp \big( \E i m * Y) \, .
\end{align*}

(2) Since $\smalln^+$ is generated by $\E i m$ for $1 \leq i \leq \ell$ and $0 < m < p$ it suffices to check that
\begin{equation}
E_0 \cdot (\E i m * X) = 0
\end{equation}
for all $X \in \barn$, $1 \leq i \leq \ell$, and $0 < m < p$. We have (using Lemma \ref{lem: basic E0 and F0 facts})
\begin{align*} E_0 \cdot (\E i m * X) &= E_0 \cdot \left( \sum_{j = 0}^m (-1)^{m-j} \E i j X \E i {m-j} \right) \\
&= X E_0 \, \E i m + \sum_{j = 1}^m (-1)^{m-j} E_0 \E i j X \E i {m-j} \\
&\quad \tr{(since $E_0$ is central in $\barn$)} \\
&=  0 \; \; \tr{(since $E_0 \cdot \smalln = 0$).}
\end{align*}
\end{proof}

\subsubsection{The morphism $\spl$} \label{subsub:spl}

Set $N := | \Delta^+ |$. For $n \geq 0$ and $\lambda \in \Lambda$ define a morphism
\begin{subequations}
\begin{equation}
\spl : \ihe n \lambda \to \ih n \lambda
\end{equation}
by
\begin{equation}  \label{eq:spl}
(\spl f)( X \otimes Y \otimes v_{n \lambda} ) = f( F_0 \cdot \varphi X \otimes E_0 \cdot \frp Y \otimes v_{pn\lambda} )
\end{equation} for all $X \in \barg$, $Y \in \barng n$, and $f \in \ihe n \lambda$.
\end{subequations}
(Here we are considering $f$ as an element of $\Indh{ \barnd \otimes \chi_{-pn\lambda} }$ under the natural inclusion). Note that $\spl$ is not a morphism of $\barg$-modules.

It is not clear that $\spl$ is well-defined, so we must prove that. We first have the following technical lemma.

\begin{lemma} \label{lem:F0 and E i pm conjugation} For all $\mu \in \Lambda$, $m \geq 0$, $1 \leq i \leq \ell$, $X \in \barg$, $Y \in \barn$, and $f \in \Indh{ \barnd \otimes \chi_{-p \mu} }$, we have $$ f \big( F_0 \, \E i {pm} X \otimes E_0 \, \frp Y \otimes v_{p \mu}  \big) = f \big( \E i {pm} F_0 \, X \otimes E_0 \, \frp Y \otimes v_{p \mu}  \big)  \, .$$
\end{lemma}

\begin{proof}
Applying the Cartan involution to Lemme 3.7 in \cite{GK10} (cf also the proof of Lemma 4.5 in \cite{KL02}) we have
\begin{equation} \label{eq:F0 and E i pm}  F_0 \, \E i {pm} \in \E i {pm} F_0 + \smalln^+ \cdot \barg + \sum_{s=0}^{m-1} \E i {sp} z_s \cdot \barg \, , \end{equation}
where $z_s \in \smallt$ are elements such that $ \chi_{ -2(p-1)\rho }(z_s) = 0 $.

Since $$ (X.f)(X' \otimes Y' \otimes v_{p \mu}) = f(X'X \otimes Y' \otimes v_{p \mu}) $$ for all $X, X' \in \barg$ and $Y' \in \barn$, it suffices to show that \begin{equation} \label{eq:F0 and E i pm 2}
f \big( F_0 \, \E i {pm} \otimes E_0 \, \frp Y \otimes v_{p \mu} \big) = f \big( \E i {pm} F_0 \, \otimes E_0 \, \frp Y \otimes v_{p \mu} \big) \, .
\end{equation} By (\ref{eq:F0 and E i pm}) we have
\begin{equation} \label{eq:F0 and E i pm 3} F_0 \, \E i {pm} \otimes E_0 \, \frp Y \otimes v_{p \mu} = \left(  \E i {pm} F_0 + \sum N_j A_j + \sum_{s=0}^{m-1} \E i {sp} z_s B_s  \right) \otimes E_0 \, \frp Y \otimes v_{p \mu} \, \end{equation} for some $A_j, B_s \in \barg$, $N_j \in \smalln^+$, and $z_s \in \smallt$ such that $ \chi_{ -2(p-1)\rho }(z_s) = 0 $.
Now,
\begin{align*}
\sum f(N_j A_j \otimes E_0 \, \frp Y \otimes v_{p \mu}) &= \sum f \Big( A_j \otimes \big( \sigma(N_j) * ( E_0 \, \frp Y \otimes v_{p \mu} ) \big) \Big) \\
&= \sum f \Big( A_j \otimes E_0 \cdot ( \sigma(N_j) * \frp Y ) \otimes v_{p \mu}  \Big) \\
& \quad \tr{(since $\sigma(N_j) * E_0 = 0$ by Lemma \ref{lem: basic E0 and F0 facts} and  $\sigma(N_j).v_{p \mu} = 0$)} \\
&=  0 \; \tr{(by Lemma \ref{lem:useful E0 and F0 facts} (2))} \, .
\end{align*}
Also, 
\begin{align*}
\sum_{s=0}^{m-1}  f \left( \E i {sp} z_s B_i \otimes E_0 \, \frp Y \otimes v_{p \mu}  \right) &= \sum_{s=0}^{m-1} f \left( B_i \otimes \sigma( \E i {sp} z_s )*( E_0 \, \frp Y \otimes v_{p \mu} )  \right) \\
&= \sum_{s=0}^{m-1}  f \left( (-1)^s B_i \otimes \sigma(z_s)* \big( ( \E i {sp} * E_0 \, \frp Y ) \otimes v_{p \mu} \big)  \right) \\
&= \sum_{s=0}^{m-1} f \left( (-1)^s B_i \otimes \big( c_{ -2(p-1)\rho }(z_s) \big). \big( \E i {sp} * E_0 \, \frp Y \big) \otimes v_{p \mu}  \right) \\
&\quad \tr{(by (\ref{eq:smallt and p weights}), since $\big( \E i {sp} * E_0 \, \frp Y \big) \otimes v_{p \mu}$ has weight} \\
& \quad \; \tr{$ 2(p-1) \rho$ mod $p \Lambda$)} \\
&= 0 \quad (\tr{since } c_{ -2(p-1)\rho }(z_s) = 0) \, .
\end{align*} Thus (\ref{eq:F0 and E i pm 2}) holds by (\ref{eq:F0 and E i pm 3}).
\end{proof}

\begin{proposition} \label{pr:well-definedness & properties of spl} The morphism $\spl$ is well-defined and divides weights by $p$ (i.e., if $f$ is a weight vector of weight $\mu$ then $\spl(f)$ is a weight vector of weight $\mu / p$ if $\mu \in p\Lambda$ and $\spl(f) = 0$ otherwise). Furthermore,\begin{equation} \label{eq:well-definedness of spl}
\spl( \varphi Z.f ) = Z. (\spl f) \tr{ for all } Z \in \barg \tr{ and } f \in \ihe n \lambda . 
\end{equation}
In particular, $\spl$ preserves $\barg$-locally finite vectors, so that $ \spl $ restricts to a morphism
\begin{equation}
\ige n \lambda \to \ig n \lambda \, .
\end{equation}
\end{proposition}

\begin{proof} To see that $\spl$ is well-defined, we need to check (cf (\ref{eq:barb equivariance})) that for $\lambda \in \Lambda$, $X \in \barg$, $ Y \in \barng n $, $Z \in \barb$, and $f \in \ihe n \lambda$,
\begin{equation} \label{eq:well-definedness & properties of spl 1}
(\spl f)(ZX \otimes Y \otimes v_{n \lambda}) = (\spl f) \big( X \otimes \sigma(Z) * ( Y \otimes v_{n \lambda} ) \big) \, . 
\end{equation}
(That is, we need to check that $\spl$ preserves $\barb$-linearity). It suffices to check this for the two cases where $Z = \displaystyle \binom {H_i} m$ or $Z = \E i m$ for some $1 \leq i \leq \ell$ and $m \geq 0$.

For the first case, set $Z = \displaystyle \binom {H_i} m$. For $1 \leq i \leq \ell$, $m \geq 0$, and $n \in \Z$ define
\begin{equation}
\binom{ H_i; n } m := \frac{ ( H_i + n ) (H_i + n - 1) \cdots ( H_i + n - m + 1 ) }{m!} \in \bart \,.
\end{equation}
We may assume in (\ref{eq:well-definedness & properties of spl 1}) that $ Y $ is a weight vector of weight $\mu$. Then we have
\begin{align*} (\spl f) \left( \binom{H_i} m X \otimes Y \otimes v_{n\lambda} \right) &= f \left( F_0 \cdot \varphi \left( \binom{H_i} m X \right) \otimes E_0 \cdot \frp Y \otimes v_{pn \lambda}  \right) \\
&= f \left( F_0 \cdot \binom{H_i}{pm} \cdot \varphi(X) \otimes E_0 \cdot \frp Y \otimes v_{pn \lambda}  \right) \\
&= f \left( \binom{H_i ; 2(p-1)}{pm} \cdot F_0 \cdot \varphi(X) \otimes E_0 \cdot \frp Y \otimes v_{pn \lambda}  \right) \\
& \quad \tr{(by \cite{L90}, 6.5(a6))} \\
&= f \left( F_0 \cdot \varphi(X) \otimes \sigma \binom{H_i ; 2(p-1)}{pm} * \Big( E_0 \cdot \frp Y \otimes v_{pn \lambda} \Big)  \right) \\
&= f \left( F_0 \cdot \varphi(X) \otimes \binom{-H_i ; 2(p-1)}{pm} * \Big( E_0 \cdot \frp Y \otimes v_{pn \lambda} \Big)  \right) \\
&= f \Bigg( F_0 \cdot \varphi(X) \otimes \binom{ -\big( 2(p-1) \rho + p \mu + pn \lambda \big)(\alpha_i^\vee) + 2(p-1) }{pm} \cdot  \\
& \quad \Big( E_0 \cdot \frp Y \otimes v_{pn \lambda} \Big) \Bigg) \\
& \quad \tr{(since $E_0 \cdot \frp Y \otimes v_{pn \lambda}$ has weight $2(p-1) \rho + p \mu + pn \lambda$)} \\
&= f \Bigg( F_0 \cdot \varphi(X) \otimes \binom{ -( p \mu + pn \lambda )(\alpha_i^\vee) }{pm} \cdot \Big( E_0 \cdot \frp Y \otimes v_{pn \lambda} \Big) \Bigg) \\
&= f \left( F_0 \cdot \varphi(X) \otimes \binom{ -( \mu + n \lambda )(\alpha_i^\vee) } m \cdot \Big( E_0 \cdot \frp Y \otimes v_{pn \lambda} \Big)  \right) \\
&= (\spl f) \left( X \otimes \binom{ -( \mu + n \lambda )(\alpha_i^\vee) } m \cdot Y \otimes v_{n \lambda}  \right)   \\
&= (\spl f) \left( X \otimes \sigma \binom{ H_i } m * (Y \otimes v_{n \lambda})  \right) \\
& \quad \tr{(since $Y \otimes v_{n \lambda}$ has weight $ \mu + n \lambda $).}
\end{align*}

For the second case, set $Z = \E i m$. Then
\begin{align*}
(\spl f) \left( \E i m X \otimes Y \otimes v_{n \lambda} \right) &= f \Big(  F_0 \cdot \varphi( \E i m X) \otimes   E_0 \, \frp Y \otimes v_{pn \lambda}  \Big) \\
&= f \Big( F_0 \, \E i {pm} \, \varphi(X) \otimes E_0 \, \frp Y \otimes v_{pn \lambda}  \Big) \\
&= f \Big( \E i {pm} F_0 \, \, \varphi(X) \otimes E_0 \, \frp Y \otimes v_{pn \lambda}  \Big) \\
& \quad \tr{(by Lemma \ref{lem:F0 and E i pm conjugation})} \\
&= f \Big(  F_0 \, \, \varphi(X) \otimes \sigma( \E i {pm} ) * ( E_0 \, \frp Y \otimes v_{pn \lambda} )  \Big) \\
&= f \Big( (-1)^m F_0 \, \, \varphi(X) \otimes E_0 \cdot \big( \E i {pm} * \frp YÊ\big) \otimes v_{pn \lambda}  \Big) \\
& \quad \tr{(by Lemma \ref{lem: basic E0 and F0 facts})} \\
&= f \Big( (-1)^m F_0 \, \, \varphi(X) \otimes E_0 \cdot \frp \big( \E i m * YÊ\big) \otimes v_{pn \lambda}  \Big) \\
& \quad \tr{(by Lemma \ref{lem:useful E0 and F0 facts} (1))} \\
&= (\spl f) \Big( X \otimes \big( \sigma( \E i m ) * Y \big) \otimes v_{n \lambda} \Big) \\
&= (\spl f) \Big(  X \otimes \sigma( \E i m )*( Y \otimes v_{n \lambda} )  \Big) \, .
\end{align*}
Hence $\splv$ is well-defined.

Note that the morphism
\begin{align*} X \otimes Y \otimes v_{n\lambda} & \mapsto F_0 \cdot \varphi X \otimes E_0 \cdot \frp Y \otimes v_{pn\lambda}
\end{align*} is the morphism dual to $\spl$. Since this morphism clearly multiplies weights by $p$, $\spl$ divides weights by $p$. Finally, (\ref{eq:well-definedness of spl}) follows from (\ref{eq:barg action}) and an easy computation.
\end{proof}

\subsubsection{Frobenius-linearity of $\spl$}

Note that by the formulas in \S\ref{subsub:background 1} we have
\begin{equation} \label{eq:fr on smallg}
\fr(X) = \epsilon(X) \tr{ for all } X \in \smallg \, .
\end{equation}

\begin{subequations}
\begin{lemma} \label{lem:commutative diagram for Frobenius-linearity} The following diagrams commute:
\begin{equation} \label{eq:commutative diagram for Frobenius-linearity 1} \xymatrix{
\barg \otimes \barg \; \; \; & \ar[l]_{ \id \otimes \varphi} \; \; \; \barg \otimes \barg & \\
&& \barg \ar[ul]_\Delta \ar[dl]^\varphi \\
\barg \otimes \barg \ar[uu]^{ \fr \otimes \id } & \ar[l]^{ \hspace{.3in} \Delta } \barg \mu_0 &
}
\end{equation} and
\begin{equation} \label{eq:commutative diagram for Frobenius-linearity 2} \xymatrix{
\barn \otimes \barn  \; \; \; & \ar[l]_{ \id \otimes \frp} \; \; \; \barn \otimes \barn & \\
&& \barn \ar[ul]_\Delta \ar[dl]^\frp \\
\barn \otimes \barn \ar[uu]^{ \fr \otimes \id } & \ar[l]^{ \hspace{.3in} \Delta } \barn  &
}
\end{equation}
\end{lemma}
\end{subequations}

\begin{proof} This is implicit in \cite{GK10} and \cite{KL02}, but we verify it directly for completeness. We first verify (\ref{eq:commutative diagram for Frobenius-linearity 1}). Since all morphisms in the diagram are multiplicative, it suffices to verify that the diagram commutes for the algebra generators $\big \{ \E i m \big \}_{m \geq 0}$, $\big \{ \FF i m \big \}_{m \geq 0}$, and $\left \{   \binom{H_i} m \right \}_{m \geq 0}$ of $\barg$. We verify this for $\E i m$:
\begin{align*} \big(  ( \fr \otimes \id ) \circ \Delta \circ \varphi  \big) \big( \E i m \big) &= \big(  ( \fr \otimes \id ) \circ \Delta \big)( \E i {pm} \mu_0 ) \\
&= (\fr \otimes \id) \bigg[ \sum_{j=0}^{pm} ( \E i j \otimes \E i {pm - j} ) \cdot \sum (\mu_0)_{(1)} \otimes (\mu_0)_{(2)} \bigg] \\
&= \sum_{j=0}^m \E i j \otimes \E i {pm - pj} \cdot \sum \fr \big( (\mu_0)_{(1)} \big) \otimes (\mu_0)_{(2)} \\
& \quad \tr{(by (\ref{eq:fr on smallg}))} \\
&= \left( \sum_{j=0}^m \E i j \otimes \E i {pm - pj} \right) \cdot ( 1 \otimes \mu_0 ) \\
&= \sum_{j=0}^m \E i j \otimes \varphi( \E i{m-j} ) \\
&= \big( (\id \otimes \varphi) \circ \Delta ) \big)( \E i m ) \, .
\end{align*} The computations for $\FF i m$ and $\binom{H_i} m$ are similar, as is the computation for (\ref{eq:commutative diagram for Frobenius-linearity 2}).
\end{proof}

\begin{proposition} \label{pr:Frobenius-linearity of spl} $\spl(f^p g) = f \cdot \spl(g)$ for all $n, m \geq 0$, $f \in \ig n \lambda$, and $g \in \ige m \lambda$.
\end{proposition}

\begin{proof} Choose $X \in \barg$ and $Y \in \barng{n+m}$. Then
\begin{align*}
\spl(f^p g)(X \otimes Y \otimes v_{(n+m)\lambda}) &= (f^p g) \big( F_0 \cdot \varphi X \otimes E_0 \cdot \frp Y \otimes v_{p(n+m)\lambda} \big) \\
&= ( \frtstar f \cdot g )\big( F_0 \cdot \varphi X \otimes E_0 \cdot \frp Y \otimes v_{p(n+m)\lambda} \big) \\
& \quad \tr{(by Proposition \ref{pr:frtstar is the pth power morphism})} \\
&= \sum f \Big[  \fr \big(  (F_0)_{(1)} ( \varphi X )_{(1)} \big) \otimes \fr \big(  (E_0)_{(1)} ( \frp Y )_{(1)} \big)\otimes v_{n \lambda}  \Big] \cdot \\
& \quad g \Big[  (F_0)_{(2)} (\varphi X)_{(2)} \otimes (E_0)_{(2)} (\frp Y)_{(2)} \otimes v_{pm \lambda}  \Big] \\
& \quad \tr{(by (\ref{eq:multiplication in ih}))} \\
&= \sum f \Big[  \fr \big(  ( \varphi X )_{(1)} \big) \otimes \fr \big(  ( \frp Y )_{(1)} \big) \otimes v_{n\lambda} \Big] \cdot \\
& \quad g \Big[  F_0 \cdot (\varphi X)_{(2)} \otimes E_0 \cdot (\frp Y)_{(2)} \otimes v_{pm\lambda} \Big] \\
& \quad \tr{(by (\ref{eq:fr on smallg}))} \\
&= \sum f \big(  \sw X 1 \otimes \sw Y 1 \otimes v_{n\lambda}  \big) \cdot g \big( F_0 \cdot \varphi X \otimes E_0 \cdot \frp Y \otimes v_{pm\lambda}  \big) \\
& \quad \tr{(by Lemma \ref{lem:commutative diagram for Frobenius-linearity})} \\
&= \big( f \cdot \spl (g) \big)( X \otimes Y \otimes v_{(n+m)\lambda} ) \, .
\end{align*}
\end{proof}

\subsection{The section $\psi_{f_+ \otimes f_-}$ and the multiplication $\mulf$} \label{sub:mul}

In this section we construct a particular section $ \psi_{f_+ \otimes f_-} \in \Indg{ \barnd } $ and define the multiplication morphism $\mulf : f \mapsto \psi_{f_+ \otimes f_-} \cdot f  $.

\subsubsection{The morphism $\bpsi$}

Set $\delta := (p-1)\rho$. Recall that the \tb{Steinberg module} for $G$, denoted $\st$, is the irreducible module of highest weight $ \delta  $. It is also a Weyl module for $G$ and is self-dual. Let
\begin{equation}
\eta : \stst \to \fp
\end{equation}
be the $G$-equivariant pairing.

Recall that we are taking the conjugation action $*$ of $\barb$ on $\barnd$. Following \cite{KLT}, define a morphism
\begin{subequations}
\begin{equation} \label{eq:bpsi}
\bpsi : \stst \to \barnd, \quad v \otimes w \mapsto \bpsi_{v \otimes w}
\end{equation}
by \begin{equation}
\bpsi_{v \otimes w}(X) = \eta( v \otimes X.w ) 
\end{equation}
\end{subequations}
for $v \otimes w \in \stst$ and $X \in \barn$. Since
\begin{equation*}
\eta( Y.v \otimes w ) = \eta( v \otimes \sigma Y.w ) \tr{ for all } v, w \in \st \tr{ and } Y \in \barg
\end{equation*}
it is easy to check that $\bpsi$ is a $\barb$-equivariant morphism.

Let
\begin{equation*}
q_{(p-1)N} : \Indg{ \barnd } \twoheadrightarrow \Indg{ \barndg{(p-1)N} } 
\end{equation*}
be the $\barg$-equivariant projection. We now define a $\barg$-equivariant morphism
\begin{equation}
\psi : \stst \to \Indg{ \barnde }, \quad v \otimes w \mapsto \psi_{v \otimes w}
\end{equation}
by the following composition:
\begin{equation} \label{eq:definition of psi}
\stst \stackrel {H^0(\bpsi)} \longrightarrow \Indg{ \barnd } \stackrel { q_{(p-1)N} } \twoheadrightarrow \Indg{ \barnde } \, .
\end{equation}

Let $\pi_{(p-1)N} : \barn \twoheadrightarrow \barne$ be the $\barb$-equivariant projection. Then, considering $\Indg{ \barnde }$ as a subspace of $ \Indg{ \barnd } $, $\psi$ is given explicitly by
\begin{equation} \label{eq:psi explicitly}
\psi_{v \otimes w}( X \otimes Y ) = \sum \eta \big( \sw X 1.v \otimes \pi_{(p-1)N}(Y). \sw X 2.w \big)
\end{equation}
for $X \in \barg$ and $Y \in \barn$. (Remark that the projections $ q_{(p-1)N} $ and $\pi_{(p-1)N}$ are necessary here because in general $\psi_{v \otimes w}$ will not be a homogeneous element of $\Indg{ \barnd }$).

\begin{lemma} \label{lem:pi and E0}
$E_0 \in \barne$.
\end{lemma}

\begin{proof}
Let $ \{ y_\beta \}_{\beta \in \Delta^+} \subseteq \barndg 1 $  be a set of weight elements of $\barnd$ that generate $\barnd$ as an $\fp$-algebra such that the weight of $y_\beta$ is $-\beta$. The ideal $\dpow Ip := \langle y_\beta^p \rangle_{\beta \in \Delta^+} $ is $\barb$-stable and the quotient algebra $ \barnd / \dpow Ip \cong \smalln^\vee $ is a $\barb$-module algebra isomorphic to the coordinate algebra of the first Frobenius kernel of $U$.

Set \begin{equation}
y_0 := \prod_{\beta \in \Delta^+} y_\beta^{p-1} \in \barnde 
\end{equation}
and let
\begin{subequations}
\begin{equation}
r : \barnd \twoheadrightarrow \smalln^\vee \twoheadrightarrow \chi_{-2 \delta} 
\end{equation}
be the $\barb$-equivariant projection dual to the morphism
\begin{equation}
\chi_{2\delta} \hookrightarrow \smalln \hookrightarrow \barn, \quad v_{2\delta} \mapsto E_0 \, .
\end{equation}
Since $r(y_0) \neq 0$ we have $y_0(E_0) \neq 0$. Hence $\pi_{(p-1)N}(E_0) \neq 0$ since $y_0 \in \barnde$.
\end{subequations}

Choose nonnegative integers $\{ m_\beta \}_{\beta \in \Delta^+}$ such that not all $m_\beta$ are equal to $ p-1 $ and set $y := \prod_{\beta \in \Delta^+} y_\beta^{m_\beta}$. To show that $E_0 \in \barne$ it suffices to show that $y(E_0) = 0$, since this would imply that $E_0$ is dual to the element $y_0$ with respect to a basis of $\barnd$ consisting of homogeneous elements.

If $y$ is not of weight $ -2\delta $ then $y(E_0) = 0$ by weight considerations, so we can assume that $y$ is of weight $ -2\delta $. Thus we have $ \sum_{\beta \in \Delta^+} m_\beta \beta = 2\delta $. Since not all $m_\beta$ are equal to $p-1$, at least one of the $m_\beta$ must be $\geq p$. (Indeed, otherwise there would be an element of $\smalln$ of weight $2\delta$ that is not in the subspace spanned by $E_0$, which is false). Thus we can write $  y = y_\gamma^p \cdot y'  $ for some $\gamma \in \Delta^+$ and we have
\begin{align*}
y(E_0) &= (y^p_\gamma \cdot y')(E_0) \\
&= ( \frstar y_\gamma \cdot y' )(E_0) \\
&= \sum y_\gamma \Big( \fr \big( (E_0)_{(1)} \big) \Big) \cdot y' \big( (E_0)_{(2)} \big) \\
&= y_\gamma(1) \cdot y'( E_0 ) \\
&= 0 \quad \tr{(since $y_\gamma(1) = 0$).}
\end{align*}
Hence $E_0 \in \barne$.
\end{proof}

In particular, we have
\begin{equation}
\pi_{(p-1)N}(E_0) = E_0 \,. 
\end{equation}

\subsubsection{The section $\psi_{f_+ \otimes f_-}$ and the multiplication $\mulf$}

Let $f_+, f_- \in \st$ be nonzero highest and lowest weight vectors, respectively. Then $F_0.f_+$ is a nonzero multiple of $f_-$ and $E_0.f_-$ is a nonzero multiple of $f_+$ (cf Exercise 2.3.E(2) in \cite{BK}).

By (\ref{eq:psi explicitly}), for $X \in \barnm$ and $Y \in \barn$ we have
\begin{equation} \label{eq:psi f+ f- explicitly}
\psi_{f_+ \otimes f_-}( X \otimes Y ) = \eta \big( X.f_+ \otimes \pi_{(p-1)N}(Y).f_- \big) \, .
\end{equation}
Thus, by rescaling $f_+$ and $f_-$ if necessary, by Lemma \ref{lem:pi and E0} we have
\begin{equation} \label{eq:eta, f+, and f-}
\psi_{f_+ \otimes f_-}( F_0 \otimes E_0 ) = \eta( F_0.f_+ \otimes \pi_{(p-1)N} (E_0).f_- ) = \eta( F_0.f_+ \otimes E_0.f_- ) = 1 \,.
\end{equation}

For all $ \lambda \in \Lambda $ and $n \geq 0$ define a morphism
\begin{equation}
\mulf : \ig n \lambda \to \Indg{ \barndg{n + (p-1)N} \otimes \chi_{ -n\lambda } }
\end{equation}
given by multiplication by the section $\psi_{f_+ \otimes f_-}$. Note that $\mulf$ is $\bart$-equivariant since $f_+ \otimes f_- \in \stst$ is an element of weight $0$.

\subsection{The splitting $ \tot $} \label{sub:tot}

\subsubsection{}
\begin{subequations}
Define an endomorphism $\tot$ of $\Rh \lambda$ as follows. Set
\begin{equation}
\tot \Big( \ih m \lambda \Big) = 0 \tr{ if } p \nmid m
\end{equation}
and for $n \geq 0$ let $\tot$ be defined on $\ih{pn}\lambda$ by the composition
\begin{align}
\ih {pn} \lambda & \stackrel \mulf \longrightarrow \ihe n \lambda \\
\nonumber & \; \; \stackrel \spl \to \ih n \lambda \, .
\end{align}
\end{subequations}
By Proposition \ref{pr:well-definedness & properties of spl}, $\tot$ descends to a morphism $\R \lambda \to \R \lambda$.

\begin{definition} \label{def:Frobenius splittings of algebras} Let $A$ be an $\fp$-algebra and $s$ an $\fp$-linear endomorphism of $A$. We say that $s$ is \textbf{Frobenius linear} if $s(a^p b) = a \cdot s(b)$ for all $a, \, b \in A$. If $s$ is a Frobenius linear endomorphism of $ A $ such that $ s(a^p) = a $ for all $a \in A$ we say that $s$ is a \tb{Frobenius splitting} of $A$.
\end{definition}

\begin{theorem} \label{th:tot is a splitting} $\tot$ is a Frobenius splitting of $\Rh \lambda$ for all $\lambda \in \Lambda$. In particular, $\tot$ descends to a Frobenius splitting of $\R \lambda$.
\end{theorem}

\begin{proof} Since $\tot$ preserves $\R \lambda$ it suffices to check that $\tot$ is a Frobenius splitting of $\Rh \lambda$. We first check that $\tot$ is Frobenius-linear. Choose $n \geq 0$ and $f \in \ih n \lambda$. For $ m $ with $ p \nmid m $ and $h \in \ih m \lambda$ we have $$ f^p h \in \ihp{ pn + m }\lambda \,. $$ Thus, since $p \nmid pn + m$, we have
\begin{equation}
\tot( f^p \cdot h ) = 0 = f \cdot \tot(h) \, .
\end{equation}

Now choose $m \geq 0$ and $g \in \ih {pm} \lambda$. Since $\mulf$ is given by section multiplication we have
\begin{equation*}
\mulf(f^p \cdot g) = f^p \cdot \mulf (g) \, .
\end{equation*} Thus, by Proposition \ref{pr:Frobenius-linearity of spl},
\begin{equation}
\tot(f^p \cdot g) = \spl \big(  f^p \cdot \mulf(g)  \big) = f \cdot \tot(g) \, .
\end{equation}
Hence $\tot$ is Frobenius-linear.

We next verify that $\tot$ is a Frobenius splitting. Since $\tot$ is Frobenius linear it suffices to show that $\tot(e) = e$, where $e \in \R \lambda$ is the unit. Now, $e \in \Indh{ \barndg 0 }$ is the element such that
\begin{equation} \label{eq:the identity in R}
e(X \otimes Y) = \epsilon(X) \epsilon(Y) 
\end{equation}
for all $X \in \barg$, $Y \in \barn$. Since $$ f( ZX \otimes Y ) = f( X \otimes \sigma Z*Y ) $$ for all $Z \in \barb$, $X \in \barg$, $Y \in \barn$, and $f \in \Indh{ \barnd }$, by the triangular decomposition of $\barg$ we can assume in the following that $X \in \barnm$. We have
\begin{align*}
\big( \tot(e) \big)(X \otimes Y) &= \big( ( \spl \circ \mulf )(e) \big)(X \otimes Y) \\
&= \big( \mulf(e) \big)( F_0 \cdot \varphi X \otimes E_0 \cdot \frp Y ) \\
&= \Big( \mu_0.\big( \mulf(e) \big) \Big)( F_0 \cdot \frpm X \otimes E_0 \cdot \frp Y ) \\
&= \big( \mulf(e) \big)( F_0 \cdot \frpm X \otimes E_0 \cdot \frp Y ) \\
& \quad \tr{(since $\mulf(e)$ has weight $0$)} \\
&= \sum \eta \Big(  (F_0)_{(1)}.( \frpm X )_{(1)}.f_+ \otimes \pi_{(p-1)N} \big( (E_0)_{(1)} \cdot ( \frp Y )_{(1)} \big).f_-  \Big) \cdot \\
& \quad e\Big( (F_0)_{(2)} \cdot ( \frpm X )_{(2)} \otimes (E_0)_{(2)} \cdot ( \frp Y )_{(2)} \Big) \\
& \quad \tr{(by (\ref{eq:multiplication in ih}) and (\ref{eq:psi f+ f- explicitly}))} \\ 
&= \eta \Big(  F_0. \frpm X.f_+ \otimes \pi_{(p-1)N} \big( E_0 \cdot \frp Y \big).f_-  \Big) \quad \tr{(by (\ref{eq:the identity in R}))} \\
&= \eta \big(  F_0.f_+ \otimes \pi_{(p-1)N} ( E_0 ).f_-  \big) \cdot \epsilon(X) \cdot \epsilon(Y) \quad \tr{(by weight considerations)} \\
&= \epsilon(X) \cdot \epsilon(Y) \quad \tr{(by (\ref{eq:eta, f+, and f-}))} \\
&= e( X \otimes Y ) \, .
\end{align*}
Hence $\tot$ is a Frobenius splitting of $\Rh \lambda$.
\end{proof}

\subsection{$\spl$ and the trace map} \label{sub:spl and the trace map}

In this section we compare $\spl$ to the local trace map. The results of this section are also crucial in the proof of Proposition \ref{pr:KLT} below. The main result in this section is Proposition \ref{pr:the trace map on Rh and tot}.

\begin{definition} \label{def:trace} For any polynomial ring $ P := \fp[z_1, \ldots, z_n] $ we have the Frobenius-linear \tb{trace map} $\trace : P \to P$ which is given on monomials as follows. Set $z_0 := z_1^{p-1} \cdots z_n^{p-1}$. Then
\begin{equation}
\trace( z_0 f^p ) = f
\end{equation}
for all $f \in P$, and if $g$ is a monomial that is not of the form $z_0 f^p$ for some $f \in P$ we set $\trace(g) = 0$. Up to a nonzero constant, $\trace$ is independent of the choice of generators $z_1, \ldots, z_n$ of $P$.
\end{definition}

\begin{remark} Consider the polynomial ring $P$ as above. For any $h \in P$ we have a Frobenius-linear endomorphism $f_h$ of $P$ given by
\begin{equation}
f_h(g) = \trace(hg) \tr{ for all } g \in P .
\end{equation}
By Example 1.3.1 in \cite{BK}, every Frobenius-linear endomorphism of $P$ is of the form $f_h$ for some $h \in P$.
\end{remark}

Let $ \{ x_\beta \}_{\beta \in \Delta^+} $ (resp. $ \{ y_\beta \}_{\beta \in \Delta^+} $) be eigenfunctions in degree 1 which generate $ \fp[\nm] $ (resp. $\fp[\n]$) as polynomial rings. By (\ref{eq:the 3 barb-module algebras}) we may also consider these as elements of $\barnmd$ (resp. $\barnd$). Set
\begin{equation}
y_0 := \displaystyle \prod_{\beta \in \Delta^+} y_\beta^{p-1} \, \tr{ and } \,
x_0 := \displaystyle \prod_{\beta \in \Delta^+} x_\beta^{p-1} \, .
\end{equation}
By the proof of Lemma \ref{lem:pi and E0}, after rescaling the $x_\beta$, $y_\beta$ if necessary we have that 
\begin{equation} \label{eq:x0, y0, F0, and E0}
x_0(F_0) = y_0(E_0) = 1 \, .  
\end{equation}
These choices of polynomial generators now give trace maps $\trp$ and $\trm$ on $\barnd$ and $\barnmd$ respectively as in Definition \ref{def:trace}.

In the case that $\lambda = 0$ we set $ \Rhz := \Rh \lambda $. In particular, identifying $\Rhz$ with the polynomial ring $\barnmd \otimes \barnd$, we obtain a trace map
\begin{equation} \label{eq:trm otimes trp}
\trm \otimes \trp : \Rhz \to \Rhz \,.
\end{equation}

\begin{subequations}
Define an endomorphism $ \splm $ of $ \barnmd $ by
\begin{equation}
(\splm f)(X) = f( F_0 \cdot \frpm X )
\end{equation}
for all $f \in \barnmd$ and $X \in \barnm$. Similarly, define an endomorphism $ \splp $ of $ \barnd $ by
\begin{equation}
(\splp g)(Y) = g( E_0 \cdot \frp Y )
\end{equation}
for all $g \in \barnm$ and $Y \in \barn$.
\end{subequations}

\begin{lemma} \label{lem:decomposition of spl}
$\spl = \splm \otimes \splp$ as endomorphisms of $\Rhz$.
\end{lemma}

\begin{proof}
Choose $X \in \barnm$, $Y \in \barn$, and $f \in \Rhz$. We need to show that
\begin{equation} \label{eq:decomposition of spl}
( \spl f )(X \otimes Y) = f( F_0 \cdot \frpm X \otimes E_0 \cdot \frp Y ) \,.
\end{equation}
Now,
\begin{align} \label{eq:spl does not need mu0}
( \spl f )( X \otimes Y ) &= f ( F_0 \cdot \varphi X \otimes E_0 \cdot \frp Y ) \\
\nonumber &= (\mu_0.f)( F_0 \cdot \frpm X \otimes E_0 \cdot \frp Y ) \,.
\end{align}
Without loss of generality we may assume that $f$ is a weight vector of weight $\mu \in \Lambda$ and that $X, Y$ are weight vectors of weight $ \mu_X $ and $\mu_Y$. Since $F_0 \cdot \frpm X \otimes E_0 \cdot \frp Y $ is a weight vector of weight $p(\mu_X + \mu_Y) \in p\Lambda$ we have
\begin{equation} \label{eq:spl does not need mu0 2}
f( F_0 \cdot \frpm X \otimes E_0 \cdot \frp Y  ) = 0 \tr{ unless } \mu = -p(\mu_X + \mu_Y) \in p\Lambda \,.
\end{equation}
In particular, if $\mu \notin p\Lambda$ then $\mu_0.f = 0$ and (\ref{eq:decomposition of spl}) follows from (\ref{eq:spl does not need mu0}) and (\ref{eq:spl does not need mu0 2}). On the other hand, if $\mu \in p\Lambda$ then $\mu_0.f = f$ and (\ref{eq:decomposition of spl}) follows from (\ref{eq:spl does not need mu0}).
\end{proof}

\begin{proposition} \label{pr:the trace map on Rh and tot} $ \splm = \trm $ and $ \splp = \trp $ as endomorphisms of $ \barnmd $ and $ \barnd $, respectively. In particular, $S = \trm \otimes \trp$ as Frobenius-linear endomorphisms of $R^\h$.
\end{proposition}

\begin{proof}
We check that $\splp = \trp$; the fact that $\splm = \trm$ follows from a similar argument. Since $\splp$ and $\trp$ are Frobenius-linear endomorphisms, they are completely determined by their values on the monomials $ \displaystyle \prod_{\beta \in \Delta^+} y_\beta^{n_\beta} $ for $ 0 \leq n_\beta < p $, so it suffices to check that the values of $\splp$ and $\trp$ on those monomials are the same.

First consider a monomial $ y := \displaystyle \prod_{\beta \in \Delta^+} y_\beta^{n_\beta} $ where $0 \leq n_\beta < p$ for all $\beta \in \Delta^+$ and $n_\beta < p-1$ for some $\beta$. Then $ \trp(y) = 0 $ by definition. On the other hand, for all $X \in \barn$ we have
\begin{equation*}
\big( \splp (y) \big)(X) = y( E_0 \cdot \frp X ) \,.
\end{equation*}
We may assume that $X$ is a weight vector. Then $E_0 \cdot \frp X$ is a weight vector of weight $\geq (p-1) \rho$ and $y$ is a weight vector of weight $ \mu_y $ with $ -(p-1) \rho < \mu_y \leq 0 $. Hence $y( E_0 \cdot \frp X ) = 0$ so that $ \trp(y) = \splp(y) $.

Next, we have $ \trp(y_0) = 1 $ by definition. On the other hand, for all $X \in \barn$ we have
\begin{align*}
\big( \splp (y_0) \big)(X) &= y_0( E_0 \cdot \frp X ) \\
&= y_0(E_0) \cdot \epsilon(X) \quad \tr{(by weight considerations)} \\
&= \epsilon(X) \quad \tr{(by (\ref{eq:x0, y0, F0, and E0}))} \\
&= 1(X) \, .
\end{align*}
Thus $\splp = \trp$.
\end{proof}

\section{Base change to $k$ and main results} \label{sec:Frobenius}
Recall that $k = \overline \F_p$. We no longer assume that all schemes are over $\F_p$. Recall that $\Gk$, $\Bk$, $T_k$, etc are the groups obtained by base-changing $G$, $B$, $T$, etc to $k$. In this section we base change the above constructions to $k$ and prove that $\co = T^*(G_k/B_k)$ is Frobenius split. 

\subsection{Review of Frobenius splitting facts} \label{subsec:Frobenius splitting facts}
\quad

In this section we review the theory of Frobenius splitting. The main references are \cite{BK} and the seminal paper \cite{MR85}.

Let $X$ be a scheme over $k$. We define a morphism $F : X \to X$ as follows: let $F$ be the identity map on points and define $F^\# : \struct{X} \rightarrow F_* \, \struct{X}$ to be the $p^{\textrm{th}}$ power map $f \mapsto f^p$. Note that although $F$ is a morphism of $\F_p$-schemes, it is not a morphism of $k$-schemes. $F$ is called the \textbf{absolute Frobenius morphism}.

\begin{definition} We say that $X$ is \textbf{Frobenius split} if there is an $\struct{X}$-linear map $\varphi : F_* \struct{X} \rightarrow \struct{X}$ such that $\varphi \circ F^\#$ is the identity map on $\struct{X}$.
\end{definition}

For any invertible sheaf $\L$ on $X$ we set 
\begin{equation} \label{eq:the affine cone over L}
R_\L := \bigoplus_{n \geq 0} \cohom 0 X {\L^n} \, .
\end{equation}
Recall the definition of a Frobenius-split algebra from Definition \ref{def:Frobenius splittings of algebras}. The following fact from \cite{BK} is the starting point for algebraic Frobenius splitting.

\begin{proposition}[\cite{BK}, Lemma 1.1.14] \label{pr:algebraic version of Frobenius splitting} Let $\L$ be an ample invertible sheaf on a complete $k$-scheme $X$. Then $X$ is Frobenius split if and only if the $k$-algebra $R_\L$ is Frobenius split.
\end{proposition}

\subsection{Splitting of $\co$} \label{subsec:splitting of co}

\subsubsection{Base change}
Set $\bargk := \barg \otimes_{\F_p} k$; we have similar definitions for $\barbk$, $\barbmk$, $\barnk$, $\barnmk$, and $\bartk$. Note that $\bargk$, $\barbk$, $\barbmk$, etc are the hyperalgebras of $G_k$, $B_k$, $B^-_k$, etc. For any $\fp$-module $M$ set $M_k := M \otimes_\fp k$. For $n \geq 0$ set
\begin{equation}
\barndgk n := \barndg n \otimes_\fp k \,,
\end{equation}
the degree-$n$ component of $\barnk$.

Note that if $ M $ is a $\barg$, $\barb$, etc module then $M_k$ is a $\bargk$, $\barbk$, etc module. For $\lambda \in \Lambda$ let $\chik \lambda$ denote the 1-dimensional $\barbk$-module corresponding to the weight $\lambda$ (equivalently, $\chik \lambda = \chi_\lambda \otimes_\fp k$).

For any $ \bartk $ (resp. $\bargk$) module $V$ we let, by a slight abuse of notation, $\finh V$ (resp. $\fing V$) denote the $\bartk$ (resp. $\bargk$) locally finite part of $V$, and we set $V^\vee := \finh V^*$.

For any $\barbk$-module $N$ set 
\begin{equation}
 \indgk {N} := \fing \, \Hom_\barbk( \bargk, N ) \, .
\end{equation}
Note that for any $\barb$-module $M$ we have a $\bargk$-module isomorphism
\begin{equation} \label{eq:indgk and indg otimes k}
\indgk{M_k} \cong \indg M \otimes_\fp k \, .
\end{equation}

\subsubsection{The splitting $\totk$ of $\co$}
Fix a regular dominant weight $\lambda \in \Lambda$ and set
\begin{equation}
R^k := \R \lambda \otimes_\fp k = \bigoplus_{n \geq 0} \Indgk{ \barndgk n \otimes \chik{-n\lambda}  } \,.
\end{equation}

For any $B_k$-module $M$ let $\L(M)$ denote the $G_k$-equivariant bundle on $G_k/B_k$ with fiber $M$. By Proposition 3.7 in \cite{APW} we have
\begin{equation} \label{eq:cohom and indg}
\Cohom 0 {G_k/B_k}{ \L(M) } \cong \Indgk{ M } .
\end{equation}

Let $\pco$ denote the projectivization of the bundle $\co$ and let $\L(\lambda)$ be the line bundle on $G_k / B_k$ corresponding to the $B_k$-module $\chik{-\lambda}$. Let
\begin{equation}
\mc Pr : \pco \to G_k / B_k 
\end{equation}
be the projection and set
\begin{equation}
\mc M := \mc Pr^* \L(\lambda) \otimes \strdiv \pco 1 \,.
\end{equation}
Recall the ring
\begin{equation}
R_{\mc M} = \bigoplus_{n \geq 0} \Cohom 0 \pco { \mc M^n }
\end{equation}
as in (\ref{eq:the affine cone over L}). By the projection formula and (\ref{eq:cohom and indg}) we have
\begin{equation} \label{eq:R and L lambda}
R_{ \mc M } \cong R^k \,.
\end{equation}
Also note that $\mc M$ is very ample on $\pco$ because it is the pullback of the very ample bundle $ \L(\lambda) \boxtimes \strdiv{ \P(\mf g) } 1 $ under the inclusion
\begin{equation}
\pco = G_k \times^{B_k} \P(\n) \, \hookrightarrow \, G_k \times^{B_k} \P( \mf g ) \, \cong \, \big( G_k / B_k \big) \times \P(\mf g) \,.
\end{equation}
By Lemma 1.1.11 in \cite{BK}, if $\pco$ is split then so is $\co$. Thus, to see that $\co$ is split, it suffices by Proposition \ref{pr:algebraic version of Frobenius splitting} and (\ref{eq:R and L lambda}) to show that $R^k$ is a Frobenius split algebra.

Let $\theta : k \to k$ be the $p^{th}$ power map and let $\theta' : k \to k$ be the $p^{th}$ root map. Set
\begin{equation}
\frtstark := \frtstar \otimes_\fp \theta: R^k \to R^k 
\end{equation}
and set
\begin{equation}
\totk := \tot \otimes_\fp \theta' : R^k \to R^k \,.
\end{equation}
Then, since $\frtstar$ is the $p^{th}$-power morphism on $\R \lambda$, $ \frtstark $ is the $p^{th}$-power morphism on $R^k$. Also, since $\tot$ is Frobenius-linear, so is $\totk$. Finally, it follows from Theorem \ref{th:tot is a splitting} that $ \totk \circ \frtstark = \id $. We summarize this discussion as follows.

\begin{theorem} $\totk$ is a Frobenius splitting of $R^k$. In particular, $\co$ is Frobenius split.
\end{theorem}

\subsubsection{Comparison with \cite{KLT}} \label{subsub:KLT}

Set $\stk := \st \otimes_{\fp} k$ and let $\eta_k : \ststk \to k$ be the duality pairing. In \cite{KLT} the authors construct, for any element $v \in \ststk$ such that $\eta_k(v) \neq 0$, a Frobenius splitting $f_v$ of $\co$. Their construction also requires them to fix a Springer isomorphism $U \isom \n$ so let us assume that the isomorphism used in their construction is the same one we fixed in \S \ref{subsub:Springer isom} above. In \S7 of \cite{KLT} they then construct, for any splitting $f_v$, a homogeneous splitting $ \pi_{(p-1)N} (f_v) $ of $\co$. (In this context "homogeneous" means that the splitting divides degrees by $p$).

Recall the highest and lowest weight elements $f_+$, $f_- \in \st$ as in \S\ref{sub:mul}. Set $f_+^k := f_+ \otimes 1 \in \stk$ and $f_-^k := f_- \otimes 1 \in \stk$.

\begin{proposition} \label{pr:KLT} The splitting of $\co$ induced by the splitting $\totk$ of $R^k$ is the same as the splitting $ \pi_{(p-1)N} (f_{f_+^k \otimes f_-^k}) $.
\end{proposition}

\begin{proof} Let $pr : \co \to G_k / B_k$ be the projection. Set $$F_k := pr^{-1}( U_k^- B_k ) \subseteq \co \,,$$ the fiber over the big cell. Then $ F_k \cong U_k^- \times \nk $. Set $\cofp := G \times^B \n$ and set $$F := U^- B \times \n \subseteq \cofp \,.$$ Then $\co = \cofp \times^\fp k$ and $F_k = F \times^\fp k$. It suffices to check that the two splittings coincide on the open set $F_k \subseteq \co$.

Denote by $\kltspl_k$ the restriction of the splitting $ \pi_{(p-1)N} (f_{f_+^k \otimes f_-^k}) $ to $F_k$. We now define a splitting $\kltspl$ of $ F $ such that $\kltspl_k$ is the base-change to $k$ (along with a twist by the $p^{th}$ root map $\theta'$) of $\kltspl$. Using our chosen Springer isomorphism we have
\begin{equation}
\fp[F] \cong \fp[U^-] \otimes \fp[U] \,.
\end{equation}
For each $m \geq 0$ let $\fp[U]_m$ denote the degree-$m$ component via the identification $\fp[U] \cong \fp[\n]$. Also recall from (\ref{eq:bpsi}) the definition of the morphism $\bpsi : \stst \to \barnd$. Using the identification $$ \Indg{ \barnd } \cong \Cohom 0 \cofp { \struct \cofp } \,, $$ we obtain a morphism
\begin{equation}
\hpsi := H^0( \bpsi ) : \stst \to \Cohom 0 \cofp { \struct \cofp } \,, \quad v \otimes w \mapsto \hpsi_{v \otimes w} \,.
\end{equation}
Now, following \cite{KLT}, $\kltspl$ is defined by the direct sum of the following compositions for $n \geq 0$:
\begin{equation} \label{eq:def of varphi}
\fp[U^-] \otimes \fp[U]_{pn} \stackrel{\cdot \, \hpsi_{f_+ \otimes f_-}} \longrightarrow \fp[U^-] \otimes \fp[U] \stackrel{ \trm \otimes \trp } \longrightarrow \fp[U^-] \otimes \fp[U] \stackrel{ q_n } \longrightarrow \fp[U^-] \otimes \fp[U]_n \,,
\end{equation}
where $ \cdot \, \hpsi _{f_+ \otimes f_-} $ denotes multiplication by the function $ \hpsi_{f_+ \otimes f_-} \in \Cohom 0 \cofp { \struct \cofp } $, $\trm \otimes \trp$ is the trace morphism as in (\ref{eq:trm otimes trp}), and $q_n$ is projection onto the $n^{th}$ homogeneous component $\fp[U^-] \otimes \fp[U]_n$. We also set
\begin{equation}
\kltspl\big(  \fp[U] \otimes \fp[U^-]_m  \big) = 0 \tr{ if } p \nmid m \,.
\end{equation}
It now suffices to verify that the splitting of $\fp[F]$ induced by $\tot$ is the same as $\kltspl$. 

Now, the splitting of $\fp[F]$ induced by the splitting $\tot$ of the ring $\Rh \lambda$ comes from the $\fp$-algebra isomorphism
\begin{equation} \label{eq:Rh lambda and fp[F]}
\Rh \lambda \cong \fp[F] 
\end{equation}
constructed as follows. First, recall that when $\lambda = 0$ we set $R^\h = \Rh \lambda$.  As in \S\ref{sub:spl and the trace map}, we have isomorphisms
\begin{equation} \label{eq:isoms for KLT}
\fp[F] \cong \fp[U^-] \otimes \fp[U] \cong \fp[\nm] \otimes \fp[\n] \cong \barnmd \otimes \barnd \cong R^\h \,.
\end{equation}
Note that for each $\lambda \in \Lambda$ there is a natural $\fp$-algebra isomorphism
\begin{subequations}
\begin{equation}
\bigoplus_{n \geq 0} \barndg n \otimes \chi_{-n\lambda} \cong \barnd
\end{equation}
which is \emph{not}, however, even $\bart$-equivariant. Thus, via the identification (\ref{eq:frtstar is the pth power morphism}), we get a natural $\fp$-algebra isomorphism
\begin{equation} \label{eq:the isomorphism r lambda}
r_\lambda :  \Rh \lambda = \bigoplus_{n \geq 0} \barnmd \otimes \barndg n \otimes \chi_{-n\lambda} \, \stackrel \sim \longrightarrow \, \bigoplus_{n \geq 0} \barnmd \otimes \barndg n = \Rhz
\end{equation}
given explicitly by
\begin{equation}
(r_\lambda f)( X \otimes Y ) = f( X \otimes Y \otimes v_{n\lambda} )
\end{equation}
\end{subequations}
for all $n \geq 0$, $X \in \barnm$, and $Y \in \barng n$. (As above, though, this is not $\bart$-equivariant). Combining (\ref{eq:isoms for KLT}) and (\ref{eq:the isomorphism r lambda}) we get the desired isomorphism (\ref{eq:Rh lambda and fp[F]}).

Now, it is easy to see that the following diagram commutes for all $\lambda$:
\begin{equation} \label{eq:the isomorphism r lambda and tot}
\xymatrix{ \Rh \lambda \ar[r]^{r_\lambda} \ar[d]_\tot & R^\h \ar[d]^\tot \\
\Rh \lambda \ar[r]_{r_\lambda} & R^\h \,.
}
\end{equation}
Also, by (\ref{eq:isoms for KLT}) we can consider $\kltspl$ as a splitting of $R^\h$. Hence it suffices to check that $\kltspl$ and $\tot$ are equal, considered as splittings of $R^\h$.

First, we have that $ \kltspl $ and $\tot$ are both zero on homogeneous elements of $R^\h$ of degree $m \nmid p$. Next, considering $ \trm \otimes \trp $ as an endomorphism of $R^\h$ as in \S\ref{sub:spl and the trace map}, by (\ref{eq:def of varphi}) we have that $\kltspl$ is given on the $pn^{th}$ homogeneous component of $R^\h$ by the following composition:
\begin{subequations}
\begin{equation}
\barnmd \otimes \barndg{pn} \stackrel{ \cdot \, \hpsi_{f_+ \otimes f_-} } \longrightarrow R^\h \stackrel{ \trm \otimes \trp } \longrightarrow R^\h\stackrel{ q_n } \longrightarrow \barnmd \otimes \barndg n \,.
\end{equation}
Here we denote, as above, the projection onto the $n^{th}$ homogeneous component of $R^\h$ by $q_n$. Since $ \trm \otimes \trp $ sends elements of degree $pn + (p-1)N$ to elements of degree $n$, this is the same as the composition
\begin{equation} \label{eq:KLT 1}
\barnmd \otimes \barndg{pn} \stackrel{ \cdot \, \hpsi_{f_+ \otimes f_-} } \longrightarrow R^\h \stackrel{ q_{pn + (p-1)N} } \longrightarrow R^\h\stackrel{ \trm \otimes \trp } \longrightarrow \barnmd \otimes \barndg n \,.
\end{equation}
\end{subequations}
On the other hand, recall that $\tot$ is given on the $pn^{th}$ homogeneous component of $R^\h$ by
\begin{equation} \label{eq:KLT 2}
\barnmd \otimes \barndg{pn} \stackrel{ \mulf } \longrightarrow R^\h \stackrel{ \spl } \longrightarrow \barnmd \otimes \barndg n \,.
\end{equation}

Now, by the definition of $ \psi $ in (\ref{eq:definition of psi}), we have that $\psi = q_{(p-1)N} \circ \hpsi$. Hence for all $ f \in \barnmd \otimes \barndg{pn} $ we have
\begin{align}
\mulf(f) &= f \cdot \psi_{f_+ \otimes f_-} \\
\nonumber &= f \cdot \Big( q_{(p-1)N} \big(  \hpsi_{f_+ \otimes f_-} \big) \Big) \\
\nonumber &= q_{pn + (p-1)N}\big(  f \cdot \hpsi_{f_+ \otimes f_-}  \big) \,.
\end{align}
Also, by Proposition \ref{pr:the trace map on Rh and tot}, $\spl = \trm \otimes \trp$. Thus (\ref{eq:KLT 1}) and (\ref{eq:KLT 2}) are the same morphism, so we have that $\kltspl$ and $\tot$ give the same splitting of $ R^\h $ as desired.
\end{proof}

\begin{remarks}
\quad
\begin{enumerate}
	\item Although the rings $\R \lambda$ are nonisomorphic for various choices of $\lambda$, the splitting of $\co$ induced by $\totk$ does not depend on the choice of regular dominant $\lambda \in \Lambda$. Indeed, $\totk$ restricts to the same splitting (\ref{eq:KLT 2}) of the open set $ F_k \subseteq \co $ regardless of the choice of $\lambda$.
	\item For a parabolic subalgebra $\mf p \supseteq \mf b$ let $\n_{\mf p}$ denote its nilradical. In \cite{MV} and \cite{vdK} it is shown that in type $A$ the splitting $ \pi_{(p-1)N} (f_{f_+^k \otimes f_-^k}) $ compatibly splits the subbundles $G_k \times^{B_k} (\n_{\mf p})_k$ for every parabolic subalgebra $\mf p \supseteq \b$. A main hope of algebraic Frobenius splitting is to extend this result to other types.
	\item Since the splitting $ \pi_{(p-1)N} (f_{f_+ \otimes f_-}) $ is $B$-canonical we have that the splitting $\totk$ is also $B$-canonical. In the algebraic context $B$-canonicity is equivalent to the fact that
\begin{equation}
\tot( \varphi Z.f ) = Z. \big( \tot f \big) \tr{ for all } f \in \R \lambda \tr{ and } Z \in \barb \,.
\end{equation}
However, I do not know how to show this directly.
	\item By Proposition 4.1.17 in \cite{BK}, if $\n$ were $B$-canonically split then one would immediately obtain a $B$-canonical splitting of $\co$ as well. Since $\co$ is $B$-canonically split, it is tempting to try to use algebraic techniques to construct a $B$-canonical splitting of $\n$. However, by the following argument due to Kumar, it is known that $\n$ is \emph{not} $B$-canonically split.
	
	Indeed, if $\n$ were $B$-canonically split, then by Exercise 4.1.E(4) in \cite{BK} $\co$ would be split compatibly with the divisor $ D := (p-1) \pi^* \partial(\flag)  \, .$ Here, $ \pi : \co \to \flag $ is the projection and $ \partial(\flag) \subseteq \flag $ is the divisor $ \bigcup_{ i = 1 }^\ell X_{w_0 s_i} $, where the $s_i \in W$ are the simple reflections, $w_0$ is the longest element of the Weyl group, and for any element $w$ of the Weyl group, $X_w := \overline{BwB} \subseteq \flag$ is the associated Schubert variety. Now, $$ \strdiv \flag D \cong \pi^* \LL( (p-1) \rho ) \,,$$ so by Lemma 1.4.7(i) of \cite{BK} we would have the following consequence: If $\lambda \in \Lambda$ is such that $\pi^* \LL(  p \lambda + (p-1) \rho  )$ has higher cohomology vanishing on $ \T^* $ then so does $ \pi^* \LL( \lambda ) $. By base change this would also be true in characteristic 0; but this is known to be false (cf \cite{Bro94}).
	\item Replacing the $*$-action of $\barb$ on $\barn$ by the multiplication action, one can construct an algebraic splitting of the affine variety $G_k/T_k \cong G_k \times^{B_k} U_k$. Note that here one does not need to use a Springer isomorphism.
\end{enumerate}
\end{remarks}

\newpage

\bibliographystyle{amsplain}

\bibliography{/users/charleshague/library/texshop/templates/bibs/thebibliography}

\end{document}

%% file: definitions.tex
\theoremstyle{plain} \numberwithin{equation}{subsection}
\newtheorem{theorem}{Theorem}[section]

\newtheorem{lemma}[theorem]{Lemma}
\newtheorem{proposition}[theorem]{Proposition}
\newtheorem{proposition/definition}[theorem]{Proposition/Definition}

\theoremstyle{definition}
\newtheorem{remark}[theorem]{Remark}
\newtheorem{remarks}[theorem]{Remarks}

\newtheorem{definition}[theorem]{Definition}

\allowdisplaybreaks[1]

\def\mf#1{{\mathfrak{#1}}} 
\def\mc#1{{   \mathcal{#1}   }}
\def\mb#1{{   \mathbb{#1}   }}

\def\LL{{ \mathcal L }}

\def\tb#1{{\textbf{#1}}}
\def\tr#1{{\textrm{#1}}}

\def\Z{{ \mb Z }}
\def\F{{ \mb F }}

\def\P{{ \mb P }}

\def\L{{ \mc L }}

\def\g{{ \mf g }}
\def\b{{ \mf b }}
\def\n{{ \mf n }}
\def\h{{ \mf h }}

\def\cohom#1#2#3{{H^{#1}  ( #2, \, #3 )}} 
\def\Cohom#1#2#3{{H^{#1} \big( #2, \: #3 \big)}} 
\def\Hom{{  \textrm{Hom}  }}

\def\struct#1{{\mathcal{O}_{#1}}} 

\def\isom{{\ \tilde{\longrightarrow} \ }} 

\def\strdiv#1#2{{  \mc{O}_{ #1 }(  #2  )  }}

\def\fr{{  \tr{Fr}  }}

\def\frp{{  \tr{Fr}'  }}
\def\frpm{{ \tr{Fr}'^- }}
\def\frpt{{ \tr{Fr}'_0 }}
\def\frstar{{  \tr{Fr}^*  }}

\def\barg{{  \bar{\tr U}( \mf{g} )  }}
\def\barb{{  \bar{\tr U}( \mf{b} )  }}
\def\barbm{{  \bar{\tr U}( \mf{b}^- )  }}
\def\barn{{  \bar{\tr U}( \mf{n} )  }}
\def\barnm{{  \bar{\tr U}( \mf{n}^- )  }}
\def\bart{{  \bar{\tr U}^0  }}

\def\bargk{{  \bar{U}_k( \mf{g} )  }}
\def\barbk{{  \bar{U}_k( \mf{b} )  }}
\def\barbmk{{  \bar{U}_k( \mf{b}^- )  }}
\def\barnk{{  \bar{U}_k( \mf{n} )  }}
\def\barnmk{{  \bar{U}_k( \mf{n}^- )  }}
\def\bartk{{  \bar{U}_k^0  }}

\def\barndgk#1{{ \bar{U}_{#1}(\n)^\vee_k }}

\def\fing{{  F_{ \mf g }  }}

\def\finh{{  F_\mf{h}  }}

\def\indhugly#1{{  \finh \, \Hom_{\barb}( \barg, #1 )  }} 

\def\indh#1{{  H^0_{\mf h}( \bar X, #1 )  }} 
\def\Indh#1{{  H^0_{\mf h} \big( \bar X, #1 \big)  }}

\def\indg#1{{  H^0( \bar X, #1 )  }} 
\def\Indg#1{{  H^0 \big( \bar X, #1 \big)  }}

\def\symnd{{ S( \mf n^* ) }}

\def\P{{ \mb P }}
\def\d#1{{ {#1}^\vee }} 

\def\barnmd{{  \d{ \barnm }  }}
\def\barnd{{ \d \barn }}

\def\bargd{{ \d \barg }}

\def\ho#1{{ H^0( #1 ) }}

\def\nm{{\n^-}}
\def\fp{{ \mb F_p }}

\def\frtstark{{ \tilde \fr^*_k }}

\def\nk{{\n_k}}

\def\bargk{{ \bar U_k(\g) }}
\def\barbk{{ \bar U_k(\b) }}
\def\barbmk{{ \bar U_k(\b^-) }}
\def\barnk{{\bar U_k( \n )}}
\def\barnmk{{\bar U_k( \nm )}}
\def\bartk{{\bar U^0_k}}

\def\Gk{{ G_k }}
\def\Bk{{ B_k }}
\def\flag{{ G_k/B_k }}
\def\T{{ \mc T }}

\newcommand\id{{ \tr{Id} }}
\newcommand\spl{{ S }} 
\newcommand\splm{{ \spl_- }}
\newcommand\splp{{ \spl_+ }}
\newcommand\splv{{ \spl^\vee }}

\newcommand\barndg[1]{{ \bar {\tr U}_{#1}( \mf n )^\vee }} 

\def\indgk#1{{  H^0_k( \bar X, #1 )  }} 
\def\Indgk#1{{  H^0_k \big( \bar X, #1 \big)  }}

\def\dpow#1#2{{ #1^{(#2)} }}
\newcommand\sw[2]{{#1_{(#2)}}} 
\newcommand\frt{{ \widetilde \fr }}

\newcommand\chik[1]{{ \chi^k_{#1} }}
\def\frtstar{{  \frt^*  }}

\newcommand\barng[1]{{ \bar {\tr U}_{#1}(\n) }}

\newcommand\ih[2]{{ \Indh{ \barndg{#1} \otimes \chi_{-#1 #2} } }}
\newcommand\ihp[2]{{ \Indh{ \barndg{#1} \otimes \chi_{-(#1) #2} } }} 
\newcommand\ihe[2]{{ \Indh{ \barndg{ (p-1)N + p#1} \otimes \chi_{-p#1 #2} } }}
\newcommand\ig[2]{{ \Indg{ \barndg{#1} \otimes \chi_{-#1 #2} } }}
\newcommand\ige[2]{{ \Indg{ \barndg{ (p-1)N + p#1 } \otimes \chi_{-p#1 #2} } }}

\newcommand\co{{\mc T^*}}
\newcommand\pco{{ \P(\co) }}

\newcommand\smalln{{ \bar{ \tr u }(\n) }}
\newcommand\smallnm{{ \bar{ \tr u }(\nm) }}
\newcommand\smallg{{ \bar{ \tr u }(\mf g) }}
\newcommand\smallt{{ \bar{ \tr u }^0 }}

\def\E#1#2{{ E_{#1}^{(#2)} }}
\def\EE#1#2{{ E_{#1}^{(#2)} }}
\def\FF#1#2{{ F_{#1}^{(#2)} }}

\def\R#1{{ R_{#1} }}
\def\Rh#1{{ R^\h_{#1} }}
\def\Rhz{{R^\h}} 

\def\isom{{ \, \stackrel \sim \to \, }}
\def\im{{ \tr{im} }}

\def\st{{ \tr{St} }}
\def\stst{{ \st \otimes \st }}
\def\stk{{ \st_k }}
\def\ststk{{ \st_k \otimes \st_k }}

\def\mul{{M}}
\def\mulf{{ \mul_{f_+ \otimes f_-} }}

\def\tot{{ \widetilde \spl }} 

\def\totk{{ \tot_k }}

\def\bpsi{{ \bar \psi }}
\def\hpsi{{ \hat \psi }}
\def\barne{{ \barng{(p-1)N} }}
\def\barnde{{ \barndg{(p-1)N} }}

\def\trace{{ \tr{Tr} }}
\def\trm{{\trace_-}}
\def\trp{{\trace_+}}

\def\cofp{{ \mc T^*_{\fp} }}
\def\kltspl{{ \Psi }}

%% file: algebraization.bbl
\providecommand{\bysame}{\leavevmode\hbox to3em{\hrulefill}\thinspace}
\providecommand{\MR}{\relax\ifhmode\unskip\space\fi MR }
\providecommand{\MRhref}[2]{%
  \href{http://www.ams.org/mathscinet-getitem?mr=#1}{#2}
}
\providecommand{\href}[2]{#2}
\begin{thebibliography}{10}

\bibitem{APW}
Henning~Haahr Andersen, Patrick Polo, and Ke~Xin Wen, \emph{Representations of
  quantum algebras}, Invent. Math. \textbf{104} (1991), no.~1, 1--59.

\bibitem{BK}
M.~Brion and S.~Kumar, \emph{Frobenius splitting methods in geometry and
  representation theory}, Progress in Mathematics, no. 231, Birkh\"auser
  Boston, 2005.

\bibitem{Bro94}
B.~Broer, \emph{Normality of some nilpotent varieties and cohomology of line
  bundles on the cotangent bundle of the flag variety}, Lie Theory and
  Geometry, Prog. Math., Birkh\"auser, 1994, pp.~1--19.

\bibitem{FP}
Eric~M. Friedlander and Brian~J. Parshall, \emph{Rational actions associated to
  the adjoint representation}, Ann. Sci. \'Ecole Norm. Sup. (4) \textbf{20}
  (1987), no.~2, 215--226.

\bibitem{Hab80}
W.~J. Haboush, \emph{Central differential operators on split semisimple groups
  over fields of positive characteristic}, S\'eminaire d'{A}lg\`ebre {P}aul
  {D}ubreil et {M}arie-{P}aule {M}alliavin, 32\`eme ann\'ee ({P}aris, 1979),
  Lecture Notes in Math., vol. 795, Springer, Berlin, 1980, pp.~35--85.

\bibitem{Ja03}
Jens~Carsten Jantzen, \emph{Representations of algebraic groups}, Mathematicals
  Surveys and Monographs, no. 107, Amer. Math. Soc., 2003.

\bibitem{GK10}
Masaharu Kaneda and Michel Gros, \emph{Contraction par {F}robenius de
  {G}-modules}, arXiv:1004.1939, 2010.

\bibitem{KLT}
Shrawan Kumar, Niels Lauritzen, and Jesper~Funch Thomsen, \emph{Frobenius
  splitting of cotangent bundles of flag varieties}, Invent. Math. \textbf{136}
  (1999), no.~3, 603--621.

\bibitem{KL00}
Shrawan Kumar and Peter Littelmann, \emph{{F}robenius splitting in
  characteristic zero and the quantum {F}robenius map}, J. Pure. Appl. Algebra
  (2000), no.~152.

\bibitem{KL02}
\bysame, \emph{Algebraization of {F}robenius splitting via quantum groups},
  Annals of Mathematics (2002), no.~155, 491--551.

\bibitem{L90}
George Lusztig, \emph{Quantum groups at roots of 1}, Geom. Dedicata \textbf{35}
  (1990), 89--113.

\bibitem{MR85}
V.~B. Mehta and A.~Ramanathan, \emph{Frobenius splitting and cohomology
  vanishing for {S}chubert varieties}, Ann. of Math. (2) \textbf{122} (1985),
  no.~1, 27--40.

\bibitem{MV}
V.~B. Mehta and Wilberd van~der Kallen, \emph{A simultaneous frobenius
  splitting for closures of conjugacy classes of nilpotent matrices},
  Compositio Mathematica \textbf{84} (1992), 211 -- 221.

\bibitem{SpUnip}
T.~A. Springer, \emph{The unipotent variety of a semi-simple group}, Algebraic
  {G}eometry ({I}nternat. {C}olloq., {T}ata {I}nst. {F}und. {R}es., {B}ombay,
  1968), Oxford Univ. Press, London, 1969, pp.~373--391.

\bibitem{vdK}
Wilberd van~der Kallen, \emph{Addendum to: A simultaneous {F}robenius splitting
  for closures of conjugacy classes of nilpotent matrices}, arXiv:0803.2960v2.

\end{thebibliography}
